\documentclass[11pt]{amsart}

\usepackage{a4wide}
\usepackage{amssymb}
\usepackage{mathrsfs}
\usepackage{enumerate}
\usepackage{esint}
\usepackage{titletoc}
\usepackage[colorlinks=true, urlcolor=blue, linkcolor=blue, citecolor=black]{hyperref}
\usepackage[nameinlink]{cleveref}

\theoremstyle{plain}
\newtheorem{theorem}{Theorem}[section]
\newtheorem{lemma}[theorem]{Lemma}
\newtheorem{corollary}[theorem]{Corollary}
\newtheorem{proposition}[theorem]{Proposition}

\theoremstyle{definition}

\newtheorem{remark}[theorem]{Remark}

\numberwithin{equation}{section}

\renewcommand{\d}{\mathrm{d}}
\newcommand{\Hn}{{\mathbb{H}^n}}
\newcommand{\Hnk}{{\mathbb{H}^n_\kappa}}

\makeatletter
\@namedef{subjclassname@2020}{\textup{2020} Mathematics Subject Classification}
\makeatother

\title[Harnack inequality for nonlocal operators on hyperbolic spaces]{Harnack inequality for fractional Laplacian-type operators on hyperbolic spaces}

\author{Jongmyeong Kim}
\address{Department of Mathematical Sciences, Seoul National University, Seoul 08826, Republic of Korea}
\email{saw2132@snu.ac.kr}

\author{Minhyun Kim}
\address{Department of Mathematics \& Research Institute for Natural Sciences, Hanyang University, 04763 Seoul, Republic of Korea}
\email{minhyun@hanyang.ac.kr}

\author{Ki-Ahm Lee}
\address{Department of Mathematical Sciences \& Research Institute of Mathematics, Seoul National University, 08826 Seoul, Republic of Korea}
\email{kiahm@snu.ac.kr}


\begin{document}

\begin{abstract}
We establish the Krylov–Safonov theory for a large class of nonlocal operators of order $2s \in (0,2)$ on hyperbolic spaces $\mathbb{H}^{n}_{\kappa}$ with curvature $-\kappa<0$. We prove the Alexandrov--Bakelman--Pucci (ABP) estimates, Krylov--Safonov Harnack inequality, and H\"older estimates. Notably, the Harnack inequality is new even for the fractional Laplacian. The novelty of the results lies in the robustness of the regularity estimates as $s \to 1$ and $\kappa \to 0$: they recover the classical regularity estimates for second-order operators on $\mathbb{H}^{n}_{\kappa}$ as $s \to 1$, and for fractional-order operators on Euclidean spaces as $\kappa \to 0$. Since the operators on hyperbolic spaces exhibit qualitatively different behavior compared to their Euclidean counterparts, we introduce new scale functions which take the effect of negative curvatures into account.
\end{abstract}

\maketitle

\section{Introduction} \label{sec:introduction}

In a celebrated series of papers \cite{CS09,CS11a,CS11b}, Caffarelli and Silvestre developed regularity theories such as Krylov--Safonov, Cordes--Nirenberg, and Evans--Krylov theories for fractional-order operators on Euclidean spaces. A key feature of their approach is that the constants in the regularity estimates remain uniform as the order of operator approaches 2. It means that the regularity theories for fractional-order and second-order operators are unified. Their results have also been extended to the parabolic setting in \cite{KL13,KL16,KL17}.

On the other hand, the regularity theory for the local operators on Riemannian manifolds has been extensively studied. In particular, the foundational work on Harnack inequalities by Yau \cite{Yau75} and Cheng--Yau \cite{CY75} has been extended to second-order operators in both divergence and non-divergence form. These are extensions of the De Giorgi–Nash–Moser \cite{SC92} and Krylov–Safonov Harnack inequalities \cite{Cab97,Kim04,WZ13}, respectively. See also \cite{KKL14,KL14} for the parabolic Harnack inequalities.

A natural next step is to establish regularity results for fractional-order operators on Riemannian manifolds. In this direction, the Harnack inequality for nonlocal operators on metric measure spaces with the volume doubling property---including Riemannian manifolds with nonnegative curvature---has been studied via the Dirichlet form theory \cite{CKW19}. However, this approach does not yield a unified regularity theory and is not suitable for operators in non-divergence form. For nonlocal operators in non-divergence form, the Krylov–Safonov Harnack inequality on Riemannian manifolds with nonnegative curvature was recently established by the authors in \cite{KKL22a}. This result unifies the Krylov–Safonov Harnack inequalities for both local and nonlocal operators on such manifolds, in the spirit of the works by Caffarelli and Silvestre.

In this paper, we further develop the unified regularity theory for fully nonlinear nonlocal operators of order $2s \in (0,2)$ on hyperbolic spaces $\mathbb{H}^{n}_{\kappa}$ with constant negative curvature $-\kappa < 0$. We establish the Alexandrov--Bakelman--Pucci (or ABP for short) estimates, Krylov--Safonov Harnack inequality, and H\"older estimates, which are robust as $s \to 1$ and $\kappa \to 0$ in the sense that the regularity estimates recover the classical ones for second-order operators on the hyperbolic spaces as $s \to 1$, and for fractional-order operators on Euclidean spaces as $\kappa \to 0$.

The operators considered in this work are modeled after the fractional Laplacian on hyperbolic spaces. Since hyperbolic geometry is not distinct from Euclidean geometry in dimension $n=1$, we assume $n \geq 2$ throughout the paper. We begin with the definition of the fractional Laplacian $-(-\Delta_{\Hnk})^{s}$ for $s \in (0, 1)$ and $\kappa > 0$ on the hyperbolic spaces $\Hnk$. Let $K_{\nu}$ denote the modified Bessel function of the second kind and define $\mathscr{K}_{\nu, a}(\rho) = \rho^{-\nu} K_{\nu}(a \rho)$ for notational convenience. The fractional Laplacian $-(-\Delta_{\Hnk})^{s}$ is defined by
\begin{equation} \label{eq:FL}
-(-\Delta_{\mathbb{H}^{n}_{\kappa}})^{s} u(x) = \, \mathrm{P.V.} \int_{\mathbb{H}^{n}_{\kappa}}  (u(z)-u(x)) \mathcal{K}_{n,s,\kappa}(d_{\Hnk}(x,z)) \,\mathrm{d}\mu_{\Hnk}(z),
\end{equation}
where the kernel $\mathcal{K}_{n, s, \kappa}$ is given by
\begin{equation} \label{eq:kernel-odd}
\mathcal{K}_{n,s,\kappa}(\rho) = c_{n, s} \sqrt{\kappa}^{1+2s} \left( \frac{-\sqrt{\kappa}\, \partial_\rho}{\sinh (\sqrt{\kappa} \rho)} \right)^{\frac{n-1}{2}} \mathscr{K}_{\frac{1+2s}{2}, \frac{n-1}{2}}\left( \sqrt{\kappa} \rho \right)
\end{equation}
when $n \geq 3$ is odd and
\begin{equation} \label{eq:kernel-even}
\mathcal{K}_{n,s,\kappa}(\rho) = c_{n, s} \int_{\rho}^{\infty} \frac{\sqrt{\kappa}^{1+2s} \sinh (\sqrt{\kappa}r)}{\sqrt{\pi}\sqrt{\cosh (\sqrt{\kappa} r) - \cosh (\sqrt{\kappa} \rho)}} \left( \frac{- \sqrt{\kappa}\,\partial_r}{\sinh (\sqrt{\kappa}r)} \right)^{\frac{n}{2}} \mathscr{K}_{\frac{1+2s}{2}, \frac{n-1}{2}}\left( \sqrt{\kappa} r \right) \, \mathrm{d}r
\end{equation}
when $n \geq 2$ is even, and
\begin{equation*}
c_{n, s} = \frac{(n-1)^{\frac{1+2s}{2}}}{2^{\frac{n-1}{2}}\pi^{\frac{n}{2}}} \frac{1}{|\Gamma(-s)|}.
\end{equation*}
See \Cref{sec:preliminaries} for details.

\begin{remark}
\begin{enumerate}[(i)]
\item
Note that the normalizing constant $c_{n, s}$ has the same asymptotic behavior with $1-s$ as $s \to 1$ up to a dimensional constant. This is a crucial fact for the robust regularity estimates as in \cite{CS09,CS11a,CS11b}.
\item
It is natural to expect that $\mathcal{K}_{n,s,\kappa}$ converges to the kernel of the fractional Laplacian on Euclidean space as curvature $-\kappa$ approaches zero. Indeed, we have from \eqref{eq:scaling-kernel} and \cite[Proposition~1.2]{KKL22b} that
\begin{equation*}
\mathcal{K}_{n,s,\kappa}(\rho) \to \rho^{-n-2s},
\end{equation*}
as $\kappa \to 0$ up to some constant depending on $n$ and $s$. 
\item
The kernel $\mathcal{K}_{n,s,\kappa}(\rho)$ decays exponentially as $\rho \to \infty$, in contrast to the polynomial decay of the kernel in Euclidean spaces. The difference comes from the exponential growth of the volume of balls in hyperbolic spaces. This is why the regularity theories on manifolds with negative curvature must be treated separately from those on manifolds with nonnegative curvatures.
\end{enumerate}
\end{remark}

Modeled on the fractional Laplacian \eqref{eq:FL}, fully nonlinear operators of the fractional Laplacian-type can be defined in the standard way. For a class $\mathcal{L}_{0}$ of linear operators of the form
\begin{equation*}
Lu(x) = \mathrm{P.V.} \int_{\Hnk} (u(z)-u(x)) \mathcal{K}(x, z) \, \d\mu_{\Hnk}(z), \quad x \in \Hnk,
\end{equation*}
with measurable kernels $\mathcal{K}$ satisfying
\begin{equation*}
\lambda \mathcal{K}_{n,s,\kappa}(d_{\Hnk}(x, z)) \leq \mathcal{K}(x, z) \leq \Lambda \mathcal{K}_{n,s,\kappa}(d_{\Hnk}(x, z)), \quad 0 < \lambda \leq \Lambda,
\end{equation*}
the {\it maximal} and {\it minimal operators} are defined by
\begin{equation*}
\mathcal{M}^{+} u(x) := \mathcal{M}^{+}_{\mathcal{L}_{0}} u(x) := \sup_{L \in \mathcal{L}_{0}} Lu(x) \quad \text{and} \quad \mathcal{M}^{-} u(x) := \mathcal{M}^{-}_{\mathcal{L}_{0}} u(x) := \inf_{L \in \mathcal{L}_{0}} Lu(x),
\end{equation*}
respectively. It is easy to see that these extremal operators are well defined at $x \in \Hnk$ for any bounded function $u$ that is $C^2$ near $x$, see \eqref{eq:well-definedness}. An operator $\mathcal{I}$ is said to be {\it elliptic with respect to $\mathcal{L}_{0}$} if
\begin{equation*}
\mathcal{M}^{-}_{\mathcal{L}_{0}} (u-v)(x) \leq \mathcal{I}(u, x) - \mathcal{I}(v, x) \leq \mathcal{M}^{+}_{\mathcal{L}_{0}} (u-v) (x)
\end{equation*}
for every point $x \in \Hnk$ and for all bounded functions $u$ and $v$ which are $C^2$ near $x$.

The first step towards the Krylov--Safonov Harnack inequality and H\"older estimate is the ABP-type estimate, which provides an estimate on the distribution function of supersolutions to fully nonlinear nonlocal operators. To state this result, we define functions
\begin{equation*}
\mathcal{S}_{\kappa}(t) = \frac{\sinh (\sqrt{\kappa}t)}{\sqrt{\kappa}t}, \quad \mathcal{H}_{\kappa}(t) = \sqrt{\kappa}t \coth (\sqrt{\kappa}t), \quad\text{and}\quad \mathcal{T}_{\kappa}(t) = \frac{\sqrt{\kappa}t}{\tanh^{-1}(\frac{1}{2}\tanh(\sqrt{\kappa}t))}.
\end{equation*}

\begin{theorem} [ABP-type estimate] \label{thm:ABP}
Let $s_{0} \in (0,1)$ and assume $s \in [s_{0},1)$. Let $u \in C^{2}(B_{5R}) \cap L^{\infty}(\Hnk)$ be a function on $\Hnk$ satisfying $\mathcal{M}^{-} u \leq f$ in $B_{5R}$, $u \geq 0$ in $\Hnk \setminus B_{5R}$, and $\inf_{B_{2R}} u \leq 1$. Let $\mathscr{C}$ be a contact set defined by \eqref{eq:contact_set}, then there is a finite collection $\lbrace Q^j_\alpha \rbrace$ of dyadic cubes, with $\mathrm{diam}(Q^j_\alpha) \leq r_0$, such that $Q^j_\alpha \cap \mathscr{C} \neq \emptyset$, $\mathscr{C} \subset \cup_j \overline{Q}^j_\alpha$, and
\begin{equation} \label{eq:Riemann_sum}
|B_R| \leq \sum_j c F^{n} |Q^j_\alpha|,
\end{equation}
where $r_0$, $c$, and $F$ are given by \eqref{eq:r_0},
\begin{equation*}
F = \mathcal{S}_{\kappa}(7R) \left( \Lambda \mathcal{H}_{\kappa}(7R) + \frac{R^2}{\mathcal{I}_{0,\kappa}(R)} \max_{\overline{Q}^j_\alpha} f \right)_{+},
\end{equation*}
and
\begin{equation*}
c = C \cosh^{n-1} \left( C \sqrt{\kappa}r_{0} \mathcal{T}_{\kappa}^2(r_0) F \right) \left( C\mathcal{T}_{\kappa}^2(r_0) F \right)^{(n-1)\log \cosh(C \sqrt{\kappa}r_{0} \mathcal{T}_{\kappa}^{2}(r_0) F)} \mathcal{T}_{\kappa}^{2n}(r_0),
\end{equation*}
respectively. See \eqref{eq:scale_functions} for the definition of $\mathcal{I}_{0,\kappa}(R)$. The universal constant $C>0$ depends only on $n$, $\lambda$, $\Lambda$, and $s_{0}$.
\end{theorem}

\begin{remark}
\begin{enumerate}[(i)]
\item
The Riemann sum in \eqref{eq:Riemann_sum} converges as $s \to 1$ to the integral
\begin{equation*}
C \int_{\mathscr{C}} \mathcal{S}_{\kappa}^{n}(7R) \left( \Lambda \mathcal{H}_{\kappa}(7R) + R^2 f(x) \right)_{+}^{n} \, \d \mu_{\Hnk}(x),
\end{equation*}
which is of the form appearing in \cite[Theorem 1.2]{WZ13}. This implies that \Cref{thm:ABP} recovers the ABP estimate for second-order operators on hyperbolic spaces as a limit $s \to 1$. Indeed, it will be proved in \Cref{prop:limit_zero} and \Cref{lem:r_0} that $\mathcal{I}_{0, \kappa}(R) \to C(n)$ and $r_{0} \to 0$ as $s \to 1$. Moreover, since $\lim_{t \to 0} \mathcal{T}_{\kappa}(t) = 2$, the dependence of $c$ on $\kappa$ and $R$ disappears in the limit $s \to 1$.
\item
\Cref{thm:ABP} provides a new result even for second-order operators because it covers fully nonlinear operators.
\item
\Cref{thm:ABP} also recovers the ABP estimate on the Euclidean spaces \cite{CS09} as $\kappa \to 0$.
\item
The function $\frac{R^2}{\mathcal{I}_{0,\kappa}(R)}$ plays the role of $R^{2s}$ in the setting of Euclidean spaces \cite{CS09} and manifolds with nonnegative curvature \cite{KKL22a}. However, it exhibits qualitatively different behavior due to the exponential decay of the kernel $\mathcal{K}_{n, s, \kappa}(R)$ for the fractional Laplacian $-(-\Delta_{\Hnk})^{s}$ as $\rho \to \infty$, in contrast to the polynomial decay in the nonnegative curvature case.
\end{enumerate}
\end{remark}

We next establish the Krylov--Safonov Harnack inequality and H\"older estimates for solutions of fully nonlinear nonlocal equations on hyperbolic spaces $\Hnk$. The Harnack inequality is new even for the fractional Laplacian on hyperbolic spaces.

\begin{theorem} [Harnack inequality] \label{thm:Harnack}
Let $s_{0} \in (0,1)$ and assume $s \in [s_{0},1)$. If a nonnegative function $u \in C^2(B_{7R}) \cap L^\infty(\mathbb{H}^n_{\kappa})$ satisfies
\begin{equation} \label{eq:main_eq}
\mathcal{M}^{-} u \leq C_{0} \quad \text{and} \quad \mathcal{M}^{+} u \geq -C_{0} \quad \text{in } B_{7R},
\end{equation}
then
\begin{equation*}
\sup_{B_{\delta_{1} R}} u \leq C \left( \inf_{B_{\delta_{1} R}} u + C_{0} \frac{(7R)^{2}}{\mathcal{I}_{0,\kappa}(7R)} \right)
\end{equation*}
for some universal constants $\delta_{1} \in (0,1)$ and $C > 0$ depending only on $n$, $\lambda$, $\Lambda$, $\sqrt{\kappa}R$, and $s_{0}$.
\end{theorem}

Let us denote by $\|\cdot\|'$ the non-dimensional norm in the following theorem.

\begin{theorem} [H\"{o}lder estimates] \label{thm:Holder}
Let $s_{0} \in (0,1)$ and assume $s \in [s_{0},1)$. If $u \in C^{2}(B_{7R}) \cap L^{\infty}(\Hnk)$ satisfies \eqref{eq:main_eq}, then
\begin{equation*}
\|u\|_{C^{\alpha}(\overline{B_{R}})}' \leq C \left( \|u\|_{L^{\infty}(\Hnk)} + C_{0} \frac{(7R)^{2}}{\mathcal{I}_{0,\kappa}(7R)} \right)
\end{equation*}
for some universal constants $\alpha \in (0,1)$ and $C > 0$ depending only on $n$, $\lambda$, $\Lambda$, $\sqrt{\kappa}R$, and $s_{0}$.
\end{theorem}

Since $\mathcal{I}_{0, \kappa}(R) \to C(n)$ as $s \to 1$ and the universal constants in \Cref{thm:Harnack} and \Cref{thm:Holder} do not depend on $s$, the regularity estimates in \Cref{thm:Harnack} and \Cref{thm:Holder} recover the classical estimates for second-order operators on hyperbolic space as limits.

The main difficulties in establishing regularity results, \Cref{thm:ABP}, \Cref{thm:Harnack}, and \Cref{thm:Holder}, arise from the influence of negative curvature. While the volume of balls in hyperbolic spaces behaves similarly to that in Euclidean spaces at small scales, it grows exponentially with increasing radius. Due to this non-homogeneity of the volume, the scaling property does not hold, making the standard arguments for regularity results break. Similar challenges are known to arise in the context of heat kernel estimates \cite{CKKW21,Gri03,GHH18}. Our results may thus offer insight into the heat kernel estimation on non-homogenous spaces. To address this difficulty, we introduce new scale functions that reflect the non-homogeneous geometry and provide their monotonicity properties in \Cref{sec:scale_functions}.

Another difficulty caused by the non-homogeneity of the volumes arises in the dyadic ring argument in the ABP estimates. In these estimates, one seeks a dyadic ring around a contact point where a supersolution remains quadratically close to a tangent paraboloid in a large portion of the ring. However, the standard Euclidean dyadic rings of the form $B_{2^{-k}r} \setminus B_{2^{-(k+1)}r}$ are no longer suitable in the framework of hyperbolic spaces. This motivates the introfuction of a {\it hyperbolic dyadic ring} whose radii are determined by the volume of balls (see \Cref{sec:ABP}). The hyperbolic dyadic ring turns out to be the natural ``dyadic" ring in hyperbolic geometry.

After the ABP estimates, we prove the regularity estimates by constructing a barrier function. It is standard to use the distance function for the construction, but the computation is significantly different from the standard one because of the hyperbolic structure. We observe in \Cref{sec:barrier} how the negative curvature of hyperbolic spaces affects the computations.

Let us also highlight some applications. As mentioned in \cite{CS09}, the fully nonlinear operators considered in this paper are naturally related to the stochastic optimal control theory \cite{Son86}. On the other hand, hyperbolic geometry plays a central role in the uniformization theorem and in the formulation of special relativity. Moreover, hyperbolic spaces have found applications in computer science, such as in hierarchical data representation and network analysis \cite{PVM21}.

This paper is organized as follows. In \Cref{sec:preliminaries}, we recall several models for hyperbolic spaces and define the fractional Laplacian $-(-\Delta_{\Hnk})^{s}$ on these spaces. In \Cref{sec:scale_functions}, new scale functions are introduced and some monotonicity properties are studied. The regularity theory begins with ABP-type estimates in \Cref{sec:ABP}. In this section \Cref{thm:ABP} is proved. The next step is the construction of a barrier function, and this is presented in \Cref{sec:barrier}. This barrier function, together with the ABP estimates, is used to obtain the so-called $L^\varepsilon$-estimates in \Cref{sec:L_eps_estimate}. \Cref{thm:Harnack} and \Cref{thm:Holder} are proved in \Cref{sec:Harnack} and \Cref{sec:Holder}, respectively. In \Cref{sec:special_functions}, some properties of special functions are collected.

\section{Preliminaries} \label{sec:preliminaries}

In this section, we recall several models for the $n$-dimensional hyperbolic spaces, revisit the definition of the fractional Laplacian on these spaces, and collect some well-known results on hyperbolic spaces.

\subsection{Hyperbolic spaces}

Let us recall the hyperboloid model and the Poincar\'e ball model (see, e.g., \cite{Geo05,Rat06,Thu97}).

We first recall the hyperboloid model
\begin{equation*}
\mathbb{H}^n_\kappa = \lbrace (x_0, \cdots, x_n) \in \mathbb{R}^{n+1}: x_0^2-x_1^2-\cdots-x_n^2 = \kappa^{-1}, x_0 > 0 \rbrace
\end{equation*}
with the metric induced by the Lorentzian metric $-\d x_0^2+\d x_1^2+\cdots+ \d x_n^2$ in $\mathbb{R}^{n+1}$. The space $\mathbb{H}^n_\kappa$ has a constant curvature $-\kappa < 0$. The internal product induced by the Lorentzian metric is denoted by $[x, x'] = x_0 x_0' - x_1 x_1' - \cdots - x_n x_n'$, and the distance between two points $x$ and $x'$ is given by
\begin{equation*}
d_{\mathbb{H}^n_\kappa}(x, x') = \frac{1}{\sqrt{\kappa}} \cosh^{-1}\left( \kappa[x, x'] \right).
\end{equation*}
Using the polar coordinates, $\mathbb{H}^n_\kappa$ can also be realized as
\begin{equation*}
\mathbb{H}^n_\kappa = \left\lbrace x=\left( \frac{\cosh r}{\sqrt{\kappa}}, \frac{\sinh r}{\sqrt{\kappa}} \omega \right) \in \mathbb{R}^{n+1}: r \geq 0, \omega \in \mathbb{S}^{n-1} \right\rbrace.
\end{equation*}
Then, the metric and volume element are given by $\d s^2 = \frac{1}{\kappa}(\d r^2 + \sinh^2 r \, \d \omega^2)$ and $\d \mu_{\mathbb{H}^n_\kappa} = \frac{1}{\sqrt{\kappa}^n}\sinh^{n-1} r \, \d r \, \d \omega$, respectively.

Let us also consider the Poincar\'e ball model $\mathbb{B}^{n}_{t, \kappa}=\lbrace y \in \mathbb{R}^n: |y|<t \rbrace$ with the metric
\begin{equation*}
\d s^2 = \frac{4b^2}{(t^2-|y|^2)^2} \, \d y^2
\end{equation*}
and the volume measure
\begin{equation} \label{eq:measure}
\d \mu_{\mathbb{B}^{n}_{t, \kappa}}(y)=\left( \frac{2b}{t^2-|y|^2} \right)^n \, \d y,
\end{equation}
where $t/b=\sqrt{\kappa}$. Note that the measure \eqref{eq:measure} tends to the Lebesgue measure $\d y$ as $\kappa \to 0$, provided that $\sqrt{\kappa}=2/t$ and $b=t^2/2$.

The map defined by
\begin{equation} \label{eq:isometry}
\phi: (x_0, x_1, \cdots x_n) \in \mathbb{H}^n_\kappa \mapsto \frac{\sqrt{\kappa}t}{1+\sqrt{\kappa}x_0}(x_1, \cdots, x_n) \in \mathbb{B}^n_{t, \kappa},
\end{equation}
or equivalently, by
\begin{equation*}
\phi: \left( \frac{\cosh r}{\sqrt{\kappa}}, \frac{\sinh r}{\sqrt{\kappa}} \omega \right) \in \mathbb{H}^{n}_{\kappa} \mapsto t \frac{\sinh \frac{r}{2}}{\cosh \frac{r}{2}} \omega \in \mathbb{B}^n_{t, \kappa},
\end{equation*}
is an isometry and its inverse is given by
\begin{equation*}
\phi^{-1}: y \in \mathbb{B}^n_{t, \kappa} \mapsto \left(  \frac{t^2+|y|^2}{\sqrt{\kappa}(t^2-|y|^2)}, \frac{2 t y_1}{\sqrt{\kappa}(t^2-|y|^2)}, \cdots, \frac{2 t y_n}{\sqrt{\kappa}(t^2-|y|^2)} \right) \in \mathbb{H}^n_\kappa.
\end{equation*}
See \cite[Chapter 8]{Geo05} for the proof. Therefore, there are several ways of describing the distance function as follows:
\begin{equation*}
\begin{split}
d_{\mathbb{B}^n_{t, \kappa}}(0,y)
&= d_{\mathbb{H}^n_\kappa}(0_\kappa,x) = \frac{1}{\sqrt{\kappa}} \cosh^{-1} \left( \kappa[0_\kappa,x] \right) \\
&= \frac{1}{\sqrt{\kappa}} \cosh^{-1} \left( \frac{t^2 + |y|^2}{t^2-|y|^2} \right) = \frac{1}{\sqrt{\kappa}} \log \left( \frac{t+|y|}{t-|y|} \right),
\end{split}
\end{equation*}
where $0_{\kappa} =(\frac{1}{\sqrt{\kappa}},0,\cdots,0) \in \mathbb{H}^n_\kappa$. We shall write $0_\kappa=0$ if there is no confusion.

\subsection{Fractional Laplacian on hyperbolic spaces}

The fractional Laplacian on $\Hn$ is defined in \cite{BGS15} by using the Helgason Fourier transform \cite{GGG03,Hel08,Ter85} (see also \cite{Geo05,Fer15}), and its normalizing constant is computed in \cite{KKL22b} by using the heat kernel \cite{GN98}. However, these works are built on hyperbolic space with curvature $-1$. Since we are working on $\Hnk$ with curvature $-\kappa<0$, we define the fractional Laplacian $-(-\Delta_{\Hnk})^{s}$ on $\Hnk$ and deduce the representation of its kernel from that of $-(-\Delta_{\Hn})^{s}$.

We recall the Helgason Fourier transform on hyperbolic spaces. The interested reader may consult \cite{Hel08,GGG03,Ter85,Geo05,Fer15}. By means of the isometry \eqref{eq:isometry}, we may work on the Poincar\'e ball model instead of the hyperboloid model. Let $t=1$ be fixed and denote $\mathbb{B}^{n}_{\kappa} = \mathbb{B}^{n}_{1, \kappa}$. The Helgason Fourier transform is defined for $u \in C^{\infty}_{c}(\mathbb{B}^{n}_{\kappa})$ by
\begin{equation*}
\widehat{u}(\lambda, \xi; \kappa) = \int_{\mathbb{B}^{n}_{\kappa}} u(x) e_{-\lambda,\xi;\kappa}(x) \, \d \mu_{\mathbb{B}^{n}_{\kappa}}(x), \quad \lambda \in \mathbb{R},~ \xi \in \mathbb{S}^{n-1},
\end{equation*}
where
\begin{equation*}
e_{\lambda, \xi; \kappa}(x) = \left( \frac{1-|x|^2}{|\xi-x|^2} \right)^{\frac{n-1}{2}+i\frac{\lambda}{\sqrt{\kappa}}}
\end{equation*}
is the eigenfunction with eigenvalue $-(\lambda^2+ \kappa \frac{(n-1)^2}{4})$. It is well known that the following inversion formula holds:
\begin{equation} \label{eq:inversion}
u(x)= \int_{-\infty}^{\infty} \int_{\mathbb{S}^{n-1}} \widehat{u}(\lambda, \xi; \kappa) e_{\lambda,\xi;\kappa}(x) \frac{\sqrt{\kappa}^{n-1}}{ |c_\kappa(\lambda)|^2} \,\d \sigma(\xi)\, \d \lambda,
\end{equation}
where
\begin{equation*}
c_\kappa(\lambda) = \sqrt{2}(2\pi)^{n/2} \frac{\Gamma(i\frac{\lambda }{\sqrt{\kappa}})}{\Gamma(\frac {n-1}{2}+i\frac{\lambda}{\sqrt{\kappa}})}
\end{equation*}
is the Harish-Chandra coefficient. Moreover, the Plancherel formula holds:
\begin{equation} \label{eq:Plancherel}
\int_{\mathbb{B}^{n}_{\kappa}} |u(x)|^2 \, \d \mu_{\mathbb{B}^{n}_{\kappa}}(x) = \int_{-\infty}^\infty \int_{\mathbb{S}^{n-1}} |\widehat{u}(\lambda, \xi;\kappa)|^2 \frac{\sqrt{\kappa}^{n-1}}{ |c_\kappa(\lambda)|^2} \,\d \sigma(\xi)\, \d \lambda.
\end{equation}

Since
\begin{equation*}
\begin{split}
\widehat{-\Delta_{\mathbb{B}^{n}_{\kappa}} u}(\lambda, \xi; \kappa)
&= - \int_{\mathbb{B}^{n}_{\kappa}} \Delta_{\mathbb{B}^{n}_{\kappa}} u(x) e_{-\lambda, \xi; \kappa}(x) \,\mathrm{d}\mu_{\mathbb{B}^{n}_{\kappa}}(x) \\
&= - \int_{\mathbb{B}^{n}_{\kappa}} u(x) \Delta_{\mathbb{B}^{n}_{\kappa}} e_{-\lambda, \xi; \kappa}(x) \,\mathrm{d}\mu_{\mathbb{B}^{n}_{\kappa}}(x) = \left( \lambda^2+\kappa\frac{(n-1)^2}{4} \right) \widehat {u}(\lambda,\xi; \kappa),
\end{split}
\end{equation*}
it is natural to define the fractional Laplacian $-(-\Delta_{\mathbb{B}^{n}_{\kappa}})^{s}$ by
\begin{equation} \label{eq:FL_Bnt}
\widehat{(-\Delta_{\mathbb{B}^{n}_{\kappa}})^{s} u}(\lambda, \xi; \kappa) = \left( \lambda^2 + \kappa \frac{(n-1)^2}{4} \right)^{s} \widehat{u}(\lambda, \xi; \kappa),
\end{equation}
which coincides with the definition given in \cite{BGS15,KKL22b} when $\kappa = 1$. By using \eqref{eq:Plancherel} and \eqref{eq:FL_Bnt}, we have
\begin{equation*}
\begin{split}
\widehat{(-\Delta_{\mathbb{B}^{n}_{\kappa}})^{s}u}(\lambda, \xi; 1)
&= \int_{\mathbb{B}^{n}_{1}} (-\Delta_{\mathbb{B}^{n}_{\kappa}})^{s} u(x) e_{-\lambda, \xi; 1}(x) \,\mathrm{d}\mu_{\mathbb{B}^{n}_{1}}(x) \\
&= \int_{\mathbb{B}^{n}_{1}} u(x) (-\Delta_{\mathbb{B}^{n}_{\kappa}})^{s}e_{-\sqrt{\kappa}\lambda, \xi; \kappa}(x) \,\mathrm{d}\mu_{\mathbb{B}^{n}_{1}}(x) \\
&= \kappa^{s} \left( \lambda^2 + \frac{(n-1)^2}{4} \right)^{s} \int_{\mathbb{B}^{n}_{1}} u(x) e_{-\sqrt{\kappa}\lambda, \xi; \kappa}(x) \,\mathrm{d}\mu_{\mathbb{B}^{n}_{1}}(x) \\
&= \kappa^{s} \left( \lambda^2+\frac{(n-1)^2}{4} \right)^{s} \widehat {u}(\lambda,\xi; 1) \\
&= \kappa^{s} \widehat{(-\Delta_{\mathbb{B}^{n}})^{s} u}(\lambda, \xi; 1),
\end{split}
\end{equation*}
and hence $(-\Delta_{\mathbb{B}^{n}_{\kappa}})^{s}u = \kappa^{s} (-\Delta_{\mathbb{B}^{n}})^{s}u$ by the inversion formula \eqref{eq:inversion}. Since
\begin{equation} \label{eq:scaling}
d_{\mathbb{B}^{n}}(x, z) = \sqrt{\kappa} d_{\mathbb{B}^{n}_{\kappa}}(x, z) \quad\text{and}\quad \mathrm{d}\mu_{\mathbb{B}^{n}}(z) = \sqrt{\kappa}^{n} \,\mathrm{d}\mu_{\mathbb{B}^{n}_{\kappa}}(z),
\end{equation}
we have
\begin{equation*}
\begin{split}
(-\Delta_{\mathbb{B}^{n}_{\kappa}})^{s}u(x)
&= \kappa^{s} \int_{\mathbb{B}^{n}} (u(x)-u(z)) \mathcal{K}_{n, s, 1}(d_{\mathbb{B}^{n}}(x, z)) \,\mathrm{d}\mu_{\mathbb{B}^{n}}(z) \\
&= \kappa^{s} \int_{\mathbb{B}^{n}_{\kappa}} (u(x)-u(z)) \mathcal{K}_{n, s, 1}(\sqrt{\kappa} \, d_{\mathbb{B}^{n}_{\kappa}}(x, z)) \sqrt{\kappa}^{n} \,\mathrm{d}\mu_{\mathbb{B}^{n}_{\kappa}}(z),
\end{split}
\end{equation*}
from which we conclude
\begin{equation} \label{eq:scaling-kernel}
\mathcal{K}_{n, s, \kappa}(\rho) = \sqrt{\kappa}^{n+2s} \mathcal{K}_{n, s, 1}(\sqrt{\kappa} \rho).
\end{equation}
Therefore, \eqref{eq:kernel-odd} and \eqref{eq:kernel-even} follow from \cite[Definition~1.1]{KKL22b} and \eqref{eq:scaling-kernel}.

\begin{proposition} \label{prop:kernel-decreasing}
$\mathcal{K}_{n, s, \kappa}$ is strictly decreasing.
\end{proposition}

\begin{proof}
It suffices to prove that $\mathcal{K}_{s} := \mathcal{K}_{n, s, 1}$ is non-increasing. It is known \cite{KKL22b} that the kernel $\mathcal{K}_{s}$ is represented as
\begin{equation*}
\mathcal{K}_{s}(\rho) = C \int_{0}^{\infty} p(t, \rho) \frac{\mathrm{d}t}{t^{1+s}}
\end{equation*}
for some $C=C(n, s)>0$, where $p(t, \rho)$ is the heat kernel of the Laplacian $\Delta_{\Hn}$ on hyperbolic space $\Hn$. Moreover, it is known \cite{AOCMM21} that $p(t, \rho)$ is strictly decreasing with respect to $\rho$. Therefore, $\mathcal{K}_{s}$ is strictly decreasing.
\end{proof}

\subsection{Hyperbolic spaces revisited}

We collect some well-known results on hyperbolic spaces. The first one is the volume doubling property that will be used frequently throughout the paper.

\begin{lemma} \label{lem:VD}
For any $B_r \subset B_R \in \mathbb{H}^n_\kappa$,
\begin{equation} \label{eq:VD}
\left( \frac{R}{r} \right)^n \leq \frac{|B_R|}{|B_r|} \leq \mathcal{D} \left( \frac{R}{r} \right)^{\log_2 \mathcal{D}},
\end{equation}
where $\mathcal{D} = 2^n \cosh^{n-1}(2\sqrt{\kappa}R)$.
\end{lemma}

\Cref{lem:VD} is a direct consequence of the Bishop--Gromov inequality. See \cite[Theorem 18.8 and Corollary 18.11]{Vil09} for the first inequality and the second inequality with $R=2r$ in \eqref{eq:VD}, respectively. For the full inequality, we find a $k \in \mathbb{N}$ satisfying $R \in [2^{k-1} r, 2^k r)$, and then iterate the inequality.

The next result is a bound of the Hessian of the squared distance. See \cite[Lemma 3.12]{CEMS01} for instance.

\begin{lemma} \label{lem:Hessian}
Fix a point $y \in \Hnk$ and consider the distance function $d_{\Hnk}(\cdot, y)$. Then,
\begin{equation*}
D^2 \left( d_\Hnk^2(x, y) / 2 \right)(\xi, \xi) \leq \mathcal{H}_{\kappa}(d_\Hnk(x, y)) |\xi|^2
\end{equation*}
for all $\xi \in T_x\Hnk$.
\end{lemma}

Let us close this section with the following generalization of Euclidean dyadic cubes that will be used in the decomposition of the contact set and in the Calder\'on--Zygmund technique.

\begin{theorem} [Christ \cite{Chr90}] \label{thm:dyadic_cubes}
There is a countable collection $\lbrace Q_\alpha^j \subset \Hnk : j \in \mathbb{Z}, \alpha \in I_j \rbrace$ of open sets and constants $c_1, c_2 > 0$ (with $2c_1 \leq c_2$), and $\delta_0 \in (0,1)$, depending only on $n$, such that
\begin{enumerate}[(i)]
\item
$|\Hnk \setminus \cup_\alpha Q_\alpha^j| = 0$ for each $j \in \mathbb{Z}$,
\item
if $i \geq j$, then either $Q_\beta^i \subset Q_\alpha^j$ or $Q_\beta^i \cap Q_\alpha^j = \emptyset$,
\item
for each $(j, \alpha)$ and each $i < j$, there is a unique $\beta$ such that $Q_\alpha^j \subset Q_\beta^i$, 
\item
$\mathrm{diam}(Q_\alpha^j) \leq c_2 \delta_0^j$, and
\item
each $Q_\alpha^j$ contains some ball $B(z_\alpha^j, c_1 \delta_0^j)$.
\end{enumerate}
\end{theorem}

The original statement of \cite[Theorem 11]{Chr90} consists of six properties. As mentioned in \cite{Chr90}, the first five properties concern only the quasi-metric space structure and the last property requires the space to be of homogeneous type. Since hyperbolic spaces are not of homogeneous type, the last property---which is not needed in this work---cannot be included.

\section{Scale functions} \label{sec:scale_functions}

Recall that, in the Euclidean spaces, it is sufficient to obtain regularity estimates in $B_{1}$ because the estimates in $B_{R}$ can be recovered from those in $B_{1}$ by using the scale invariance of the equations. The scale invariance heavily relies on the homogeneity of the underlying spaces. However, hyperbolic spaces are not homogeneous. Indeed, the volume of a ball grows exponentially as the radius goes to infinity in hyperbolic spaces, and hence the kernel $\mathcal{K}_{n, s, \kappa}(\rho)$ of the fractional Laplacian $(-\Delta_{\mathbb{H}^{n}_{\kappa}})^{s}$ decays exponentially as $\rho \to \infty$. Therefore, it is crucial to find appropriate scale functions that capture the correct behavior at every scale. For this purpose, we define
\begin{align} \label{eq:scale_functions}
\begin{split}
\mathcal{I}_{0,\kappa}(R) &:= \int_{B_R(x)} d_{\Hnk}^2(z, x) \mathcal{K}_{n, s,\kappa}(d_{\Hnk}(z, x)) \, \d \mu_{\Hnk}(z), \\
\mathcal{I}_{\infty, \kappa }(R) &:= \int_{\Hnk \setminus B_R(x)} R^2 \mathcal{K}_{n, s,\kappa}(d_{\Hnk}(z, x)) \, \d \mu_{\Hnk}(z),
\end{split}
\end{align}
and $\mathcal{I}_{\kappa} (R):= \mathcal{I}_{0,\kappa}(R) + \mathcal{I}_{\infty, \kappa}(R)$. They play a role of scale functions in the regularity estimates on hyperbolic spaces as the homogeneous polynomial $R^{2-2s}$ does on the Euclidean spaces.

We omit the subscript $\kappa$ when $\kappa=1$, i.e., $\mathcal{I}_{0} = \mathcal{I}_{0, 1}$, $\mathcal{I}_{\infty} = \mathcal{I}_{\infty, 1}$ and $\mathcal{I} = \mathcal{I}_{1}$. Then, it follows from \eqref{eq:scaling} and \eqref{eq:scaling-kernel}
\begin{equation} \label{eq:I-scaling}
\begin{split}
\mathcal{I}_{0,\kappa}(R)
&= \sqrt{\kappa}^{-2+2s}\mathcal{I}_{0}(\sqrt{\kappa}R), \\
\mathcal{I}_{\infty,\kappa}(R)
&= \sqrt{\kappa}^{-2+2s}\mathcal{I}_{\infty}(\sqrt{\kappa}R), \\
\mathcal{I}_{\kappa}(R)
&= \sqrt{\kappa}^{-2+2s}\mathcal{I}(\sqrt{\kappa}R).
\end{split}
\end{equation}
It is known that $\mathcal{K}_{n, s, 1}(\rho) \sim \rho^{-n-2s}$ as $\rho \to 0^{+}$ and $\mathcal{K}_{n, s, 1}(\rho) \sim \rho^{-1-s} e^{-(n-1)\rho}$ as $\rho \to \infty$, see \cite[Theorem 2.4]{BGS15} and \cite[Proposition 1.2]{KKL22b}. Thus, the scale functions in \eqref{eq:scale_functions} are well defined. Moreover, we observe
\begin{equation} \label{eq:well-definedness}
|Lu(x)| \leq \Lambda \| u \|_{C^2(\overline{B_R(x)})} \mathcal{I}_{0,\kappa}(R) + 2\Lambda \|u\|_{L^{\infty}(\Hnk)} \mathcal{I}_{\infty,\kappa}(R) < +\infty
\end{equation}
for any operator $L \in \mathcal{L}_0$ and any function $u \in C^2(\overline{B_R(x)}) \cap L^\infty(\Hnk)$.

Let us now investigate some properties of the scale functions that are useful for the upcoming regularity theory. Although the scale functions do not have scaling properties, they satisfy some monotonicity properties. For instance, $\mathcal{I}_{0, \kappa}$ is increasing and $\mathcal{I}_{\infty, \kappa}$ is decreasing by definition. Moreover, some variations of these scale functions satisfy {\it almost monotonicity}. We say that a function $f$ is {\it almost decreasing} if there exists a constant $C \geq 1$ such that $f(R) \leq Cf(r)$ for all $R>r>0$.

\begin{proposition} \label{prop:monotonicity}
The functions $R^{-2+s} \mathcal{I}_{0, \kappa}(R)$ and $R^{-2}\mathcal{I}_{0, \kappa}(R)$ are almost decreasing.
\end{proposition}

The almost monotonicity of $R^{-2}\mathcal{I}_{0, \kappa}(R)$ follows from that of $R^{-2+s}\mathcal{I}_{0, \kappa}(R)$. Furthermore, it is enough to show that $R^{-2+s}\mathcal{I}_{0}(R)$ is almost decreasing by the relation \eqref{eq:I-scaling}. Indeed, it is a direct consequence of the following lemma.

\begin{lemma} \label{lem:decreasing}
There exist constants $0< C_{1}(n) \leq C_{2}(n)$ and a non-increasing function $F$ such that $C_{1} F(R) \leq R^{-2+s} \mathcal{I}_0(R) \leq C_2 F(R)$ for all $R > 0$.
\end{lemma}

\begin{proof}
It is enough to show
\begin{equation} \label{eq:comp-f}
C_3 f(R) \leq \frac{\mathcal{I}_0'(R)}{R^{1-s}} \leq C_4 f(R)
\end{equation}
for some constants $0<C_3(n) \leq C_4(n)$ and some non-increasing function $f(R)$. Indeed, once we prove \eqref{eq:comp-f}, we then have
\begin{equation*}
\frac{C_{3}}{2} F(R) \leq \frac{\mathcal{I}_0(R)}{R^{2-s}} = \frac{\int_0^R \mathcal{I}_0'(\rho) \,\mathrm{d}\rho}{\int_0^R (2-s)\rho^{1-s} \,\mathrm{d}\rho} \leq C_4 F(R),
\end{equation*}
where
\begin{equation*}
F(R) = \frac{\int_0^R \rho^{1-s}f(\rho)\, \mathrm{d}\rho}{\int_0^R \rho^{1-s} \mathrm{d}\rho}
\end{equation*}
is a non-increasing function.

To prove the claim \eqref{eq:comp-f}, we observe
\begin{equation*}
\frac{\mathcal{I}_{0}'(R)}{R^{1-s}} = |\mathbb{S}^{n-1}| R^{1+s} \mathcal{K}_{n, s, 1}(R) \sinh^{n-1}R.
\end{equation*}
By \Cref{lem:kernel-asymp}, it is comparable to
\begin{equation*}
f(R) := R I_{\frac{n}{2}-1}\left( \frac{n-1}{2} R \right) K_{\frac{n}{2}+s}\left( \frac{n-1}{2} R \right)
\end{equation*}
up to dimensional constants, where $I_{\nu}$ is the modified Bessel function of the first kind. It only remains to show that $f$ is non-increasing. We note that it is sufficient to prove that $g(R) = R I_{\frac{n}{2}-1}(R)K_{\frac{n}{2}+s}(R)$ is non-increasing. Indeed, by using \eqref{eq:derivative} and \cite[Theorem 3]{Seg11} we obtain
\begin{equation*}
\begin{split}
g'(R)
&= - s I_{\frac{n}{2}-1}K_{\frac{n}{2}+s}+ R I_{\frac{n}{2}}K_{\frac{n}{2}+s} -R I_{\frac{n}{2}-1}K_{\frac{n}{2}+s -1} \\
&\leq \left( -s +\frac{R^2}{\sqrt{R^2+a^2}+a} - \frac{R^2}{\sqrt{R^2+(a+s)^2}+a+s} \right) I_{\frac{n}{2}-1}K_{\frac{n}{2}+s} \\
&= \left( \sqrt{R^2+a^2} - \sqrt{R^2+(a+s)^2} \right) I_{\frac{n}{2}-1}K_{\frac{n}{2}+s} \\
&\leq 0,
\end{split}
\end{equation*}
where $a=(n-1)/2$.
\end{proof}

The following result shows a relation between two scale functions $\mathcal{I}_0$ and $\mathcal{I}_\infty$. Let us mention that the function $\mathcal{H}(t)=t\coth t$ naturally appears when the negative curvature is involved as in \Cref{lem:Hessian}. The relation between $\mathcal{I}_{0, \kappa}$ and $\mathcal{I}_{\infty, \kappa}$ follows from \Cref{prop:I_infty_I_0} and \eqref{eq:I-scaling}.

\begin{proposition} \label{prop:I_infty_I_0}
There exists a constant $C=C(n) > 0$ such that
      \begin{equation*}
        \mathcal{I}_\infty(R) \leq C \frac{1-s}{s} \mathcal{H}(R) \mathcal{I}_0(R)
      \end{equation*}
      for all $R > 0$.
    \end{proposition}

\begin{proof}
By \Cref{lem:kernel-asymp}, we have
\begin{equation*}
\begin{split}
\mathcal{I}_{\infty}(R)
&= |\mathbb{S}^{n-1}| \int_{R}^{\infty} R^{2} \mathcal{K}_{n, s, 1}(\rho) \sinh^{n-1}\rho \,\mathrm{d}\rho \\
&\leq C R^{2} \int_{R}^{\infty} \rho^{-s} I_{\frac{n}{2}-1}(a\rho) K_{\frac{n}{2}+s}(a\rho) \,\mathrm{d}\rho \\
&\leq C R^{2} \int_{aR}^{\infty} \rho^{-s} I_{\frac{n}{2}-1}(\rho) K_{\frac{n}{2}+s}(\rho) \,\mathrm{d}\rho,
\end{split}
\end{equation*}
where $a=(n-1)/2$. Similarly, we also have
\begin{equation*}
\mathcal{I}_{0}(R) \geq C \int_{0}^{aR} \rho^{2-s} I_{\frac{n}{2}-1}(\rho) K_{\frac{n}{2}+s}(\rho) \,\mathrm{d}\rho.
\end{equation*}
Thus, it is enough to show that there is a constant $C=C(n)>0$ such that
\begin{equation*}
R^{2} \int_{aR}^{\infty} \rho^{-s} I_{\frac{n}{2}-1}(\rho) K_{\frac{n}{2}+s}(\rho) \,\mathrm{d}\rho \leq C \frac{1-s}{s} \mathcal{H}(R) \int_{0}^{aR} \rho^{2-s} I_{\frac{n}{2}-1}(\rho) K_{\frac{n}{2}+s}(\rho) \,\mathrm{d}\rho.
\end{equation*}
By \Cref{lem:A}, the problem is reduced to finding a constant $C=C(n)>0$ such that
\begin{equation} \label{eq:A-I}
\begin{split}
&I_{\frac{n}{2}-1}(aR) K_{\frac{n}{2}+s}(aR) + I_{\frac{n}{2}}(aR) K_{\frac{n}{2}+s-1}(aR) \\
&\leq C \mathcal{H}(R) \bigg[ I_{\frac{n}{2}-1}(aR) K_{\frac{n}{2}+s}(aR) + I_{\frac{n}{2}}(aR) K_{\frac{n}{2}+s-1}(aR) \\
&\quad\quad\quad - \frac{1}{2-s} \left( I_{\frac{n}{2}}(aR) K_{\frac{n}{2}+s-1}(aR) + I_{\frac{n}{2}+1}(aR) K_{\frac{n}{2}+s-2}(aR) \right) \bigg].
\end{split}
\end{equation}
Indeed, it is easy to check that \eqref{eq:A-I} holds by comparing the asymptotic behavior of the functions on both sides.
\end{proof}

Since the limit behavior of the scale functions as $s \to 0$ are of interest in the unified regularity theory, we recall the following results.

\begin{proposition} \label{prop:limit_zero}
There exists a constant $C = C(n)> 0$ such that
\begin{equation*}
\lim_{s \to 1} \mathcal{I}_0(R) = \lim_{s \to 1} \mathcal{I}_{0, \kappa}(R) = C
\end{equation*}
for any $R > 0$.
\end{proposition}

\begin{proof}
The assertion follows from \cite[Lemma 5.2]{KKL22b} and \eqref{eq:I-scaling}.
\end{proof}

\begin{proposition}
For any $R>0$,
\begin{equation*}
\lim_{s \to 1} \mathcal{I}_\infty(R) = \lim_{s \to 1} \mathcal{I}_{\infty, \kappa}(R) = 0.
\end{equation*}
\end{proposition}

\begin{proof}
The desired result follows from \Cref{prop:I_infty_I_0}, \Cref{prop:limit_zero}, and \eqref{eq:I-scaling}. See also \cite[Lemma 5.1]{KKL22b}.
\end{proof}
 
In the work of Caffarelli and Silvestre \cite{CS09}, the quantity $r_0 = \rho_0 2^{-1/(2-2s)}R$, which is characterized by the relation $(r_0/\rho_0)^{2-2s} = R^{2-2s}/2$, plays a fundamental role. The most important feature of this quantity is that it converges to 0 as $s \to 1$.

We define such a quantity in a similar way in our framework. Since the scale function $\mathcal{I}_0$ is strictly increasing, its inverse exists. Thus, for a given $R > 0$, we define $r_0 \in (0, R)$ by
\begin{equation} \label{eq:r_0}
r_0 = \rho_0 \mathcal{I}_{0, \kappa}^{-1} (\mathcal{I}_{0, \kappa}(R)/2)
\end{equation}
for some universal constant $\rho_0 \in (0,1)$ that will be determined later. Let us close the section with the following lemma.

\begin{lemma} \label{lem:r_0}
For a fixed $R > 0$, $\lim_{s \to 1} r_0 = 0$.
\end{lemma}

\begin{proof}
Suppose that $\lim_{s \to 1} r_0 \neq 0$. Then, since $\tilde{r} := \limsup_{s \to 1} r_0 \neq 0$, there is a sequence $s_k \to 1$ such that $r_{0, k} := \rho_0 \mathcal{I}_{0,\kappa, s_{k}}^{-1}(\mathcal{I}_{0,\kappa, s_{k}}(R)/2) \in (0, R)$ converges to $\tilde{r}$ as $k \to \infty$, where $\mathcal{I}_{0, \kappa, s_{k}}$ is the scale function $\mathcal{I}_{0, \kappa}$ with respect to $s_{k}$. We have
\begin{equation} \label{eq:r_0k}
\mathcal{I}_{0, \kappa, s_{k}}(r_{0, k}/\rho_0) = \mathcal{I}_{0, \kappa, s_{k}} (R)/2.
\end{equation}
By \Cref{prop:limit_zero} and continuity of $\mathcal{I}_0$, the left-hand side of \eqref{eq:r_0k} converges to $C$ as $k \to \infty$ whereas the right-hand side of \eqref{eq:r_0k} converges to $C/2$, which is a contradiction.
\end{proof}

\section{Discrete ABP-type estimates} \label{sec:ABP}

In this section, we provide the proof of \Cref{thm:ABP}. Throughout the section, $u$ is assumed to be a supersolution given in \Cref{thm:ABP}. On Riemannian manifolds, the distance squared function in construction of envelope was suggested by Cabr\'e in \cite{Cab97} and has been used by many in \cite{Kim04,KKL14,KL14,WZ13}. More precisely, for each $y \in B_R$, there is a unique paraboloid
\begin{equation*}
P_y(z) = c_y - \frac{1}{2R^2} d_\Hnk^2(z, y)
\end{equation*}
that touches $u$ from below, with a contact point $x \in B_{5R}$. The {\it envelope $\Gamma$ of $u$} is defined by
\begin{equation*}
\Gamma(z) = \sup_{y \in B_R} P_y(z),
\end{equation*}
and the {\it contact set} is given by
\begin{equation} \label{eq:contact_set}
\mathscr{C} = \lbrace x \in B_{5R}: u(x) = \Gamma(x) \rbrace.
\end{equation}

The first step toward to ABP-type estimates for nonlocal operators is to find a ring around a given contact point in which supersolution $u$ is quadratically close to the paraboloid in a large portion of the ring. In the Euclidean spaces \cite{CS09}, or more generally in Riemannian manifolds with nonnegative curvature \cite{KKL22a}, the standard dyadic rings $B_{2^{-k}r_0} \setminus B_{2^{-(k+1)}r_0}$ are used. However, these are not appropriate within the framework of hyperbolic spaces due to the lack of homogeneity of the volume of balls. We thus define $r_{k}$ recursively by
\begin{equation*}
\frac{|B_{r_k}|}{|B_{r_{k-1}}|} = 2^{-n}, \quad k=1,2,\cdots,
\end{equation*}
and a {\it hyperbolic dyadic ring} by $R_k = R_k(x) = B_{r_k}(x) \setminus B_{r_{k+1}}(x)$. Note that we have $|B_{r_k}| / |B_{r_{k-1}}| \leq (r_k/r_{k-1})^n$  from \Cref{lem:VD}, and hence $r_{k+1} \geq r_k/2$. By using the hyperbolic dyadic rings, we will prove a series of lemmas to deduce \Cref{thm:ABP}. For notational convenience, we shall write
\begin{equation*}
\tilde{f}_{\kappa}(x) := \Lambda \mathcal{H}_{\kappa}(7R) + \frac{R^{2}}{\mathcal{I}_{0, \kappa}(R)} f(x),
\end{equation*}
where we recall $\mathcal{H}_{\kappa}(t) = \sqrt{\kappa} t \coth(\sqrt{\kappa} t)$.

\begin{lemma} \label{lem:good_ring}
Let $u$ be a supersolution given in \Cref{thm:ABP}. Then, there exists a universal constant $C_0 > 0$, independent of $s$, $\kappa$, and $R$, such that for each $x \in \mathscr{C}$ and $M_0 > 0$, there is an integer $k \geq 0$ satisfying
\begin{equation} \label{eq:G_k}
|G_k| \leq \frac{C_0}{M_0} \tilde{f}_\kappa(x) |R_k|,
\end{equation}
where $G_k = R_k \cap \lbrace u > P_y + M_0 (r_k/R)^2 \rbrace$.
\end{lemma}

\begin{proof}
Since $x$ minimizes the function $u + \frac{1}{2R^2} d_{\Hnk}^2(\cdot, y)$, we have $\mathcal{M}^- u(x) \geq I_1+I_2+I_3$, where
\begin{align*}
I_1 &= \lambda \int_{B_R(x) \cup B_{5R}} \delta\left( u+\frac{1}{2R^2} d_{\Hnk}^2(\cdot, y), x,z \right) \mathcal{K}_{n, s, \kappa}(d_{\Hnk}(z,x)) \, \d\mu_{\Hnk}(z), \\
I_2 &= - \Lambda \int_{B_R(x) \cup B_{5R}} \delta^+ \left( \frac{1}{2R^2}d_{\Hnk}^2(\cdot, y), x, z \right) \mathcal{K}_{n, s, \kappa}(d_\Hnk(z, x)) \, \d \mu_{\Hnk}(z), \\
I_3 &= - \Lambda \int_{\Hnk \setminus (B_R(x) \cup B_{5R})} \delta^-(u,x,z) \mathcal{K}_{n, s, \kappa} (d_\Hnk(z, x)) \, \d \mu_{\Hnk}(z),
\end{align*}
and $\delta(v, x, z) = (v(z)+v(\exp_x(-\exp_x^{-1}z)) - 2v(x))/2$ is the second order incremental quotients. By the mean value theorem for integrals and \Cref{lem:Hessian}, we obtain
\begin{equation} \label{eq:I_2_ring}
I_2 \geq - C \Lambda \mathcal{H}_\kappa(7R) \frac{\mathcal{I}_{0,\kappa}(R)}{R^2}.
\end{equation}
Since $u(x) \leq u(x) +\frac{1}{2R^2} d_{\Hnk}^2(x,y) \leq \inf_{B_{2R}} (u+\frac{1}{2R^2} d_{\Hnk}^2(x,y) \leq 11/2 < 6$ and $u \geq 0$ in $\Hnk \setminus B_{5R}$, we also have
\begin{equation} \label{eq:I_3_ring}
I_3 \geq -C \Lambda \frac{\mathcal{I}_{\infty,\kappa}(R)}{R^2}.
\end{equation}

Let us now focus on $I_1$. Assume that \eqref{eq:G_k} does not hold for all $k$. Then, since
\begin{equation*}
\delta\left(u+\frac{1}{2R^2} d_{\Hnk}^2(\cdot, y),x,z\right) \geq M_0 \left(\frac{r_k}{R}\right)^2 \quad\text{on} ~ G_k
\end{equation*}
and $\mathcal{K}_{s, \kappa}$ is decreasing by \Cref{prop:kernel-decreasing}, we have
\begin{equation*}
I_1 \geq \lambda M_0 \sum_{k=1}^{\infty} \int_{G_k} \left( \frac{r_k}{R} \right)^2 \mathcal{K}_{n, s, \kappa}(d_{\Hnk}(z, x)) \, \d\mu_{\Hnk}(z) \geq \lambda C_0 \frac{\tilde{f}(x)}{R^2} \sum_{k=1}^{\infty} r_k^2 \mathcal{K}_{n, s, \kappa}(r_k) |R_k|.
\end{equation*}
Since $r_{k+1} \geq r_k/2$ and $|R_k| = 2^n|R_{k+1}|$, we obtain
\begin{align*}
I_1 
&\geq \lambda C_0 \frac{\tilde{f}(x)}{R^2} \sum_{k=0}^\infty r_{k+1}^2 \mathcal{K}_{n, s, \kappa}(r_{k+1}) |R_{k+1}| \\
&\geq \lambda 2^{-2-n} C_0 \frac{\tilde{f}(x)}{R^2} \sum_{k=0}^\infty r_k^2 \mathcal{K}_{n, s, \kappa}(r_{k+1}) |R_k| \\
&\geq \lambda 2^{-2-n} C_0 \frac{\tilde{f}(x)}{R^2} \sum_{k=0}^\infty \int_{R_k} d_{\Hnk}^2(z,x) \mathcal{K}_{n, s, \kappa}(d_\Hnk(z,x)) \, \d \mu_{\Hnk}(z) \\
&= \lambda 2^{-2-n} C_0 \frac{\tilde{f}(x)}{R^2} \mathcal{I}_{0, \kappa}(r_0).
\end{align*}
Furthermore, by using \Cref{prop:monotonicity} we have
\begin{equation} \label{eq:I_1_ring}
I_1 \geq \lambda 2^{-2-n} C_0 \frac{\tilde{f}(x)}{R^2} \rho_0^2 \mathcal{I}_{0, \kappa}(r_0/\rho_0) = \lambda 2^{-3-n} C_0 \frac{\tilde{f}(x)}{R^2} \rho_0^2 \mathcal{I}_{0,\kappa}(R).
\end{equation}
By combining \eqref{eq:I_2_ring}, \eqref{eq:I_3_ring}, and \eqref{eq:I_1_ring}, and then using \eqref{eq:r_0} and \Cref{prop:I_infty_I_0}, we obtain
\begin{equation*}
f(x) \geq \lambda 2^{-3-n} C_0 \frac{\tilde{f}(x)}{R^2} \rho_0^2 \mathcal{I}_{0,\kappa}(R) - C\left( 1+s_0^{-1} \right) \Lambda \mathcal{H}_\kappa(7R) \frac{\mathcal{I}_{0,\kappa}(R)}{R^2}.
\end{equation*}
By taking $C_0$ sufficiently large, we arrive at a contradiction.
\end{proof}

The next lemma shows that the function $\Gamma-P_y$ is $c$-convex with an appropriate function $c$. See \cite{GM96,MTW05} for the definition of $c$-convex function. The proof is exactly the same with that of \cite[Lemma 3.4]{KKL22a} except for the Hessian bound of the distance squared function. That is, we use \Cref{lem:Hessian} instead of \cite[Lemma 2.1]{KKL22a}.

\begin{lemma} \label{lem:convexity}
Let $x \in \mathscr{C}$, $z \in \Hnk$, and $y \in B_R$ be a vertex point of a paraboloid $P_y$. Then,
\begin{equation*}
(\Gamma - P_y)(z) \leq (1-t)(\Gamma - P_y) (z_1) + t (\Gamma - P_y)(z_2) + \frac{1}{2R^2} t(1-t) \mathcal{H}_\kappa( d_{\Hnk}(y,z)+|\xi| ) | \xi |^2
\end{equation*}
for all $t \in (0,1)$, where $z_1 = \exp_z(t \xi)$ and $z_2 = \exp_z((1-t)(-\xi))$.
\end{lemma}

Using \Cref{lem:convexity}, we show that the envelope is captured in a small ball near a contact point by two paraboloids that are quadratically close to each other.

\begin{lemma} \label{lem:flatness}
Under the setting of \Cref{lem:good_ring}, there is a universal constant $\varepsilon_0 \in (0,1)$ such that if
\begin{equation*}
|\lbrace z \in R_k : \Gamma(z) > P_y(z) + h \rbrace| \leq \varepsilon_0 |R_k|,
\end{equation*}
then
\begin{equation*}
\Gamma \leq P_y + h + C \mathcal{H}_\kappa(7R) \left( \frac{r_k}{R} \right)^2
\end{equation*}
in $B_{\tilde{r}_{k+1}}(x)$, where $\tilde{r}_{k+1} = \frac{1}{\sqrt{\kappa}} \tanh^{-1}(\frac{1}{2}\tanh (\sqrt{\kappa} r_{k+1}))$.
\end{lemma}

\begin{proof}
Let us fix $z \in B_{\tilde{r}_{k+1}}(x)$ and set $D = \lbrace z \in R_k: \Gamma(z) \leq P_y(z) + h \rbrace$. For $w \in R_k$, let us consider a geodesic $c:\mathbb{R}\to \Hnk$ passing through $w$ and $z$. Then, $c(\mathbb{R}) \cap R_k$ consists of two connected components $c(t_1,t_2)$ and $c(t_3,t_4)$, where $t_1<t_2<t_3<t_4$ satisfy $t_4-t_3=t_2-t_1$. We may assume that $w=c(t) \in c(t_1,t_2)$. We define a map $\varphi_z:R_k \to R_k$ by $\varphi_z(w) = c\left( -t+t_1+t_4 \right)$, which is clearly one-to-one and onto.

Among all the geodesics passing through the point $z$, let us consider geodesics $c_\perp$ that are perpendicular to the geodesic joining $x$ and $z$. Then $\cup \, c_\perp$ divides $R_k$ into two regions: let $A_1$ be the smaller one and $A_2$ the bigger one. We claim
\begin{equation} \label{eq:claim_volume}
|E| \leq |\varphi_z(E)| \quad\text{for any Borel set}~ E \subset A_1.
\end{equation}
Indeed, we may assume that $z=0_{\kappa} \in \Hnk$ by using a global isometry. Then the map $\varphi:=\varphi_z$ can be represented by
\begin{equation*}
\varphi(w) = \frac{1}{\sqrt{\kappa}}(\cosh(r+C_\theta), \sinh(r+C_\theta)(-\theta)), \quad w = \frac{1}{\sqrt{\kappa}}(\cosh r, \sinh r \theta),
\end{equation*}
where $C_\theta = d_{\Hnk}(\varphi(w^\ast),0)-d_{\Hnk}(w^\ast,0)$, with $w^\ast$, the intersection point of $\partial B_{r_{k+1}}(x)$ and the geodesic segment joining 0 and $w$. Note that $\varphi$ is a smooth map because it is a composition of smooth maps. Clearly, $C_\theta \geq 0$ if and only if $w \in A_1$. Thus, we obtain
\begin{align*}
|E|
  &= \iint {\bf 1}_{E}(w) \frac{\sinh^{n-1} (\sqrt{\kappa}r)}{\sqrt{\kappa}^{n-1}} \, \d r \, \d \theta \\
  &\leq \iint {\bf 1}_{\varphi(E)}(\varphi(w)) \frac{\sinh^{n-1}(\sqrt{\kappa}(r+C_\theta))}{\sqrt{\kappa}^{n-1}} \, \d r \,\d\theta \\
  &= \iint {\bf 1}_{\varphi(E)} \left( \frac{\cosh \tilde{r}}{\sqrt{\kappa}}, \frac{\sinh\tilde{r}}{\sqrt{\kappa}} \tilde{\theta} \right) \frac{\sinh^{n-1}(\sqrt{\kappa} \tilde{r})}{\sqrt{\kappa}^{n-1}} \, \d \tilde{r} \, \d \tilde{\theta} = |\varphi(E)|,
\end{align*}
where we have used change of variables $\tilde{r} = r+C_\theta$ and $\tilde{\theta} = -\theta$. This proves \eqref{eq:claim_volume}.

We next claim that
\begin{equation} \label{eq:claim_ratio}
|R_k| \leq C|A_1| ~ \text{with}~ C > 0 ~\text{a universal constant.}
\end{equation}
Let us first deduce the lemma assuming that \eqref{eq:claim_ratio} is true. If we show that $\varphi_z(A_1 \cap D) \cap D \neq \emptyset$, then there are points $w_i \in A_i \cap D$, $i=1,2$, such that $\varphi_z(w_1) = w_2$. Since $\Gamma(w_i) \leq P_y(w_i) + h$, for $i = 1, 2$, the desired result follows from \Cref{lem:convexity}. Assume to the contrary that the set $\varphi(A_1 \cap D) \cap D$ is empty. By \eqref{eq:claim_ratio}, we have
\begin{equation*}
|A_1 \cap D^c| \leq |R_k \cap D^c| \leq \varepsilon_0 |R_k| \leq C\varepsilon_0 |A_1|.
\end{equation*}
By taking $\varepsilon_0=(2C)^{-1}$, we obtain $|A_1 \cap D| > |A_1|/2$. Since $\varphi_z(A_1 \cap D) \subset A_2 \cap D^c$, it follows that
\begin{equation*}
\frac{1}{2} |A_1| < |A_1 \cap D| \leq |\varphi_z(A_1\cap D)| \leq |A_2 \cap D^c| \leq |R_k \cap D^c| \leq \frac{1}{2} |A_1|,
\end{equation*}
which is a contradiction.

From now on, we focus on the proof of \eqref{eq:claim_ratio}. To this end, it is convenient to use the Poincar\'e ball model $\mathbb{B}^n_{\kappa} = \mathbb{B}^{n}_{1, \kappa}$. Let $\tilde{A}_1 = \phi(A_1)$ and $\tilde{R}_k = \phi(R_k)$, where $\phi$ is the isometry given by \eqref{eq:isometry}. Since we are concerned with volumes, we may assume $\phi(z) = |\phi(z)|e_1$ so that $\tilde{A}_1$ is rotationally symmetric with respect to $x_1$-axis. Let $\rho_k$ be such that $r_k= d_{\mathbb{B}^{n}_{\kappa}} (0,\rho_k e_1)$. We observe that
\begin{equation} \label{eq:annulus_sector}
\lbrace y \in \mathbb{B}^n_{\kappa}: \rho_{k+1} < |y| < \rho_k,~ e_1 \cdot y/|y| > 1/2 \rbrace \subset \tilde{A}_1.
\end{equation}
Indeed, if we define $\tilde{A}_1'$ in the same way as $\tilde{A}_1$ with $z' \in \partial B_{\tilde{r}_{k+1}}(x)$ instead of $z \in B_{\tilde{r}_{k+1}}(x)$, then $\tilde{A}_1 \supset \tilde{A}_1'$. Moreover, any geodesic that is perpendicular to $x_1$-axis and passes through $\tilde{\rho}_{k+1}e_1$ is contained in the sphere
\begin{equation} \label{eq:sphere}
\bigg( x_1- \frac{1+\tilde{\rho}_{k+1}^2}{2\tilde{\rho}_{k+1}} \bigg)^2 + x_2^2+x_3^2+ ...+ x_n^2 = \bigg( \frac{1-\tilde{\rho}_{k+1}^2}{2\tilde{\rho}_{k+1}} \bigg)^2.
\end{equation}
The $x_1$-coordinate of the intersection of the spheres \eqref{eq:sphere} and $x_1^2+x_2^2+...+x_n^2=\rho_{k+1}^2$ is given by
\begin{equation*}
x_1=\frac{\tilde{\rho}_{k+1}}{1+\tilde{\rho}_{k+1}^2}(1+\rho_{k+1}^2) = \frac{\tanh (\sqrt{\kappa}\tilde{r}_{k+1})}{\tanh (\sqrt{\kappa} r_{k+1})} \rho_{k+1} = \frac{1}{2} \rho_{k+1},
\end{equation*}
where we used
\begin{equation*}
  d_{\Hnk}(0, \rho e_1) = \frac{1}{\sqrt{\kappa}}\cosh^{-1} \frac{1+\rho^2}{1-\rho^2} =\frac{1}{\sqrt{\kappa}} \tanh^{-1} \frac{2\rho}{1+\rho^2}
\end{equation*}
in the second equality. Note that the radius $\tilde{r}_{k+1} = \frac{1}{\sqrt{\kappa}} \tanh^{-1} (\frac{1}{2} \tanh (\sqrt{\kappa}r_{k+1}))$ is chosen so that the last equality holds. Therefore, \eqref{eq:annulus_sector} holds.

We now compute
\begin{align*}
  |A_1| & = |\tilde{A}_1| \\
        & \geq \int_0^{2\pi} \int_{\rho_{k+1}}^{\rho_k} \int_0^{\frac{\pi}{3}} \cdots \int_0^{\frac{\pi}{3}}\left( \frac{2}{\sqrt{\kappa}(1-\rho^2)} \right)^{n} \rho^{n-1} \sin^{n-2} \varphi_1   \cdots \sin \varphi_{n-2}      \, \d \varphi_1 \cdots \d  \varphi_{n-2} \, \d \rho \, \d \theta \\
        & \geq C(n) \int_{\rho_{k+1}}^{\rho_k} \left( \frac{2}{\sqrt{\kappa}(1-\rho^2)} \right)^{n} \rho^2 \, \d\rho.
\end{align*}
Since
\begin{equation*}
|R_k| = |\tilde{R}_k| = |S^{n-1}|\int_{\rho_{k+1}}^{\rho_k} \left( \frac{2}{\sqrt{\kappa}(1-\rho^2)} \right)^{n} \rho^{n-1} \, \d \rho,
\end{equation*}
\eqref{eq:claim_ratio} is proved with some $C(n)>0$.
\end{proof}

We define $\phi: \Hnk \to B_R$ by a map assigning each point $x \in \Hnk$ a vertex point $y$ of the paraboloid $P_y$, where $P_y$ is a paraboloid such that $\Gamma(x) = P_y(x)$, which is not necessarily unique. Then, the flatness of $\Gamma$ obtained in \Cref{lem:flatness} allows us to control the image of $\phi$, which can be understood as the image of gradient of $\Gamma$.

\begin{lemma} \label{lem:gradient_map}
Under the setting of \Cref{lem:good_ring}, let $x \in \mathscr{C}$ and let $k$ be such that \eqref{eq:G_k} holds, and let $\varepsilon_0$ be the constant in \Cref{lem:flatness}. Then,
\begin{equation} \label{eq:ring_eps}
\left| \left\lbrace z \in R_k: u(z) > P_y(z) + C\tilde{f}_\kappa(x) (r_k/R)^2 \right\rbrace \right| \leq \varepsilon_0 |R_k|
\end{equation}
and
\begin{equation} \label{eq:gradient_map}
\phi \left( \overline{B(x, \tilde{r}_{k+1}/2)} \right) \subset B\left( y, C \mathcal{S}_{\kappa}(7R)\mathcal{T}_{\kappa}(r_{k+1}) \tilde{f}_\kappa(x) r_k \right),
\end{equation}
where $C > 0$ is a universal constant depending only on $n$, $\lambda$, $\Lambda$, and $s_0$.
\end{lemma}

\begin{proof}
By taking $M_0 = C_0 \tilde{f}_\kappa(x)/\varepsilon_0$ in \Cref{lem:good_ring}, we obtain \eqref{eq:ring_eps}. Moreover, by \Cref{lem:flatness} we have
\begin{equation} \label{eq:flatness}
P_y \leq \Gamma \leq P_y + C \tilde{f}(x) \left( \frac{r_k}{R} \right)^2
\end{equation}
in $B_{\tilde{r}_{k+1}}(x)$, with a universal constant $C > 0$.

To prove \eqref{eq:gradient_map}, let $z \in \overline{B(x,\tilde{r}_{k+1}/2)}$ and $y_\ast \in \phi(z)$. We need to find a upper bound of $d_{\Hnk}(y_\ast, y)$. Let $\xi_1 = \exp_z^{-1} y_\ast$ and $\xi_2 = \exp_z^{-1} y$. Let us consider a family of geodesics
\begin{equation*}
c(s,t) = \exp_z (t(\xi_1 + s(\xi_2-\xi_1))),
\end{equation*}
and the Jacobi field $J$ along $c$. Then, by \cite[Equation (1.8b)]{Gig12} (or see, e.g. \cite{Jos17}), we have
\begin{equation*}
|J(1)|_{g(y_\ast)} \leq \mathcal{S}_{\kappa}(|\xi_1|) |J'(0)|_{g(z)} \leq \mathcal{S}_{\kappa}(7R) |\xi_2 - \xi_1|_{g(z)}.
\end{equation*}
Considering the curve $s \mapsto c(s, 1)$, we obtain
\begin{equation*}
d_{\Hnk}(y_\ast, y) \leq \int_0^1 |\partial_s c(s, 1)|_{g(y_\ast)} \, \d s \leq \mathcal{S}_{\kappa}(7R) |\exp_z^{-1}y_\ast - \exp_z^{-1} y|_{g(z)}.
\end{equation*}
By the Gauss lemma, we know that $|\exp_z^{-1}y_\ast - \exp_z^{-1} y|_{g(z)} = R^2 |\nabla P_{y_\ast}(z) - \nabla P_y(z)|_{g(z)}$. Thus, it only remains to show that
\begin{equation} \label{eq:grad}
R^2 |\nabla P_{y_\ast} (z) - \nabla P_y (z)|_{g(z)} \leq C \mathcal{T}_{\kappa}(r_{k+1}) \tilde{f}(x) r_k
\end{equation}
for some universal constant $C > 0$.

To this end, we prove that
\begin{equation} \label{eq:directional_grad}
\left| \left. \frac{\d}{\d t} \right|_{t=0} \left( P_{y_\ast} - P_y \right)(c(t)) \right| \leq C \mathcal{T}_{\kappa}(r_{k+1}) \tilde{f}_\kappa(x) \frac{r_k}{R^2}
\end{equation}
for all geodesics $c$, with unit speed, starting from $c(0) = z$. Suppose that \eqref{eq:directional_grad} does not hold for some $c$. We may assume that
\begin{equation*}
C\mathcal{T}_{\kappa}(r_{k+1}) \tilde{f}_\kappa(x) \frac{r_k}{R^2} \leq \left. \frac{\d}{\d t} \right|_{t=0} \left( P_{y_\ast} - P_y \right)(c(t)),
\end{equation*}
by considering $\tilde{c}(t) = c(-t)$ instead of $c(t)$ if necessary. Let $\varepsilon > 0$, then there is a $\delta > 0$ such that if $| t | < \delta$, we have
\begin{equation} \label{eq:t}
C \mathcal{T}_{\kappa}(r_{k+1}) \tilde{f}_\kappa(x) \frac{r_k}{R^2} - \varepsilon \leq \frac{(P_{y_\ast} - P_y)(c(t)) - (P_{y_\ast} - P_y)(c(0))}{t} \leq \frac{h(t) - h(0)}{t},
\end{equation}
where $h(t) = (\Gamma - P_y)(c(t))$. Let $T > 0$ be the first time such that $c(T) \in \partial B_{3\tilde{r}_{k+1}/4}(x)$. Let $N$ be the least integer not smaller than $T/\delta$, and let $0 = t_0 < t_1 < \cdots < t_N = T$ be equally distributed times. Then, $t_{i+1} - t_i = T/N \leq \delta$. By \Cref{lem:convexity}, we have
\begin{equation} \label{eq:t_i}
\frac{h(t_i) - h(t_{i-1})}{t_i - t_{i-1}} \leq \frac{h(t_{i+1}) - h(t_i)}{t_{i+1} - t_i} + \frac{\mathcal{H}_\kappa(7R)}{2R^2} (t_{i+1} - t_{i-1}), \quad i = 1, 2, \cdots, N-1.
\end{equation}
Thus, it follows from \eqref{eq:t} and \eqref{eq:t_i} that
\begin{equation} \label{eq:i}
C \mathcal{T}_{\kappa}(r_{k+1}) \tilde{f}_\kappa(x) \frac{r_k}{R^2} - \varepsilon \leq \frac{h(t_{i+1}) - h(t_i)}{T/N} + \frac{\mathcal{H}_\kappa(7R)}{2R^2} \frac{2Ti}{N}, \quad i=1,2,\cdots, N-1.
\end{equation}
Summing up \eqref{eq:i} for $i=1,2,\cdots, N-1$, we obtain
\begin{equation*}
N \left( C \mathcal{T}_{\kappa}(r_{k+1}) \tilde{f}_\kappa(x) \frac{r_k}{R^2} - \varepsilon \right) \leq \frac{h(t_N) - h(t_0)}{T/N} + \frac{\mathcal{H}_\kappa(7R)}{2R^2} \frac{2T}{N} \frac{N(N-1)}{2}.
\end{equation*}
Since $c$ has a unit speed, we have $\tilde{r}_{k+1}/4 < T < \tilde{r}_{k+1}$, and hence
\begin{align*}
C \mathcal{T}_{\kappa}(r_{k+1}) \tilde{f}_\kappa(x) \frac{r_k}{R^2} - \varepsilon
&\leq \frac{(\Gamma - P_y)(c(T)) - (\Gamma - P_y)(z)}{\tilde{r}_{k+1}/4} + \frac{\mathcal{H}_\kappa(7R)}{2R^2} \tilde{r}_{k+1} \\
&\leq \frac{(\Gamma - P_y)(c(T))}{\tilde{r}_{k+1}/4} + \frac{\mathcal{H}_\kappa(7R)}{2R^2} \tilde{r}_{k+1}.
\end{align*}
Recalling that $\mathcal{T}_{\kappa}(r_{k+1}) = r_{k+1}/\tilde{r}_{k+1}$ and $r_{k+1} \geq r_k/2$, and that $\varepsilon$ was arbitrary, we have
\begin{equation} \label{eq:not_flat}
C\tilde{f}_\kappa(x) \frac{r_k^2}{8R^2} \leq (\Gamma - P_y) (c(T)) + \mathcal{H}_\kappa(7R) \frac{r_k^2}{8R^2}.
\end{equation}
Since $c(T) \in \partial B_{3\tilde{r}_{k+1}/4}(x) \subset B_{\tilde{r}_{k+1}}(x)$, the inequality \eqref{eq:not_flat} with a sufficiently large constant $C_1 > 0$ contradicts to \eqref{eq:flatness}. Therefore, we have proved \eqref{eq:grad}, which finishes the proof.
\end{proof}

We are now ready to prove a discrete ABP-type estimate, from which \Cref{thm:ABP} follows.

\begin{lemma} \label{lem:ABP_discrete}
Assume the same assumptions as in \Cref{thm:ABP}. There is a finite collection $\mathcal{D}$ of dyadic cubes $\lbrace Q^j_\alpha \rbrace$, with diameters $d_j \leq r_0$, such that the following hold:
\begin{enumerate}[(i)]
\item
Any two different dyadic cubes in $\mathcal{D}$ do not intersect.
\item
$\mathscr{C} \subset \bigcup \overline{Q}^j_\alpha$.
\item
$|\phi(\overline{Q}^j_\alpha)| \leq c F^{n} |Q_\alpha^j|$.
\item
$|B(z_\alpha^j, 2r_0) \cap \lbrace u \leq \Gamma + C(\sup_{\overline{Q}_\alpha^j} \tilde{f}(x)) (r_0/R)^2 \rbrace| \geq \mu |Q^j_\alpha|$.
\end{enumerate}
The constants $C > 0$ and $\mu > 0$ depend only on $n$, $\lambda$, $\Lambda$, and $s_0$.
\end{lemma}

\begin{proof}
Let $c_1$, $c_2$, and $\delta_0$ be the constants given in \Cref{thm:dyadic_cubes}, which depend only on $n$. Let us fix the smallest integer $N$ such that $c_2 \delta_0^N \leq r_0$, then there are finitely many dyadic cubes $Q_\alpha^N$ of generation $N$ such that $\overline{Q}_\alpha^N \cap \mathscr{C} \neq \emptyset$ and $\mathscr{C} \subset \cup_\alpha \overline{Q}_\alpha^N$. Whenever a dyadic cube $Q_\alpha^j$ $(j \geq N)$ does not satisfy (iii) and (iv), we consider all of its successors $Q_\beta^{j+1} \subset Q_\alpha^j$ instead of $Q_\alpha^j$. Among these successors of $j+1$ generation, we only keep those whose closures intersect $A$ and discard the rest. We prove that this process must finish in a finite number of steps.

Assume to the contrary that the process produces an infinite sequence of nested dyadic cubes $\lbrace Q^j_\alpha \rbrace_{j=N}^\infty$. Then, the intersection of their closures is some contact point $x \in \mathscr{C}$. By \Cref{lem:gradient_map}, there is a $k \geq 0$ such that \eqref{eq:ring_eps} and \eqref{eq:gradient_map} hold. Let $j \geq N$ be such that $\delta_0 \tilde{r}_{k+1}/2 \leq c_2 \delta_0^j < \tilde{r}_{k+1}/2 \leq r_0$, then it follows from \Cref{thm:dyadic_cubes} that
\begin{equation} \label{eq:Q_in_B}
B(z_\alpha^j, c_1 \delta_0^j) \subset Q_\alpha^j \subset \overline{Q}_\alpha^j \subset B(x, \tilde{r}_{k+1}/2).
\end{equation}
Thus, it follows from \eqref{eq:gradient_map} and \eqref{eq:Q_in_B} that
\begin{equation*}
|\phi ( \overline{Q}_\alpha^j )| \leq |\phi(\overline{B(x, \tilde{r}_{k+1}/2)})| \leq |B(y, C\mathcal{S}_{\kappa}(7R) \mathcal{T}_{\kappa}(r_{k+1}) \tilde{f}_\kappa(x) r_k)|.
\end{equation*}
Since $\mathcal{S}_{\kappa}(7R) \tilde{f}_\kappa(x) \leq F$ and $r_k \leq 2r_{k+1} = 2\mathcal{T}_{\kappa}(r_{k+1}) \tilde{r}_{k+1} \leq 4\mathcal{T}_{\kappa}(r_0)c_2 \delta_0^{j-1}$, we have
\begin{equation*}
|\phi ( \overline{Q}_\alpha^j )| \leq |B(z_\alpha^j, C\mathcal{T}_{\kappa}^2(r_0) F c_1 \delta_0^j)|.
\end{equation*}
Therefore, by \Cref{lem:VD} we obtain
\begin{equation*}
|\phi ( \overline{Q}_\alpha^j )| \leq \mathcal{D} \left( C \mathcal{T}_{\kappa}^{2}(r_0) F \right)^{\log_2 \mathcal{D}} |Q_\alpha^j|
\end{equation*}
where $\mathcal{D} = 2^n \cosh^{n-1} (C\sqrt{\kappa} \mathcal{T}_{\kappa}^{2}(r_0) c_1\delta_0^j F)$, which shows that $Q_\alpha^j$ satisfies (iii).

If $z \in B(x, r_k)$, then $d(z, z_\alpha^j) \leq d(z, x) + d(x, z_\alpha^j) < r_k + c_2 \delta_0^j \leq 2r_0$, which shows that $B(x, r_k) \subset B(z_\alpha^j, 2r_0)$. Thus, by using \eqref{eq:ring_eps}, we have
\begin{align*}
|B(z_\alpha^j, 2r_0) \cap \lbrace u \leq \Gamma + C(\sup_{\overline{Q}_\alpha^j} \tilde{f}_\kappa(x)) (r_0/R)^2 \rbrace|
&\geq |R_k \cap \lbrace u \leq P_y + C\tilde{f}_\kappa(x)(r_k/R)^2 \rbrace | \\
&\geq (1-\varepsilon_0) |R_k| \\
&= (1-\varepsilon_0) (2^n-1)|B_{r_{k+1}}| \\
&\geq \mu |Q_\alpha^j|
\end{align*}
for some universal constant $\mu > 0$. This proves that $Q_\alpha^j$ also satisfies (iv), which yields a contradiction. Therefore, the process must stop in a finite number of steps.
\end{proof}

\section{A barrier function} \label{sec:barrier}

This section is devoted to a construction of a special barrier function, which is a key ingredient together with the ABP-type estimates for the Krylov–Safonov Harnack inequality. It is standard to use distance function to construct a barrier function, but computations are significantly different from the standard argument. We will observe how the negative curvature of hyperbolic spaces comes into play. Let us begin with some inequalities.

\begin{lemma} \label{lem:arccos}
Let $\alpha > 0$ and $R_0 > 0$. Then
\begin{equation} \label{eq:arccos1}
(\cosh^{-1}(t\cosh (\sqrt{\kappa} R_0)))^{-2\alpha} - (\sqrt{\kappa}R_0)^{-2\alpha} \geq -2\alpha \frac{\mathcal{H}_\kappa(R_0)}{(\sqrt{\kappa} R_0)^{2\alpha+2}} (t-1)
\end{equation}
for all $t > 1/\cosh (\sqrt{\kappa}R_0)$. Moreover,
\begin{align} \label{eq:arccos2}
\begin{split}
  &\frac{(\cosh^{-1}(t \cosh (\sqrt{\kappa}R_0)))^{-2\alpha-2}}{t^2 \cosh^2 (\sqrt{\kappa}R_0) - 1} - \frac{(\sqrt{\kappa}R_0)^{-2\alpha-2}}{\sinh^2 (\sqrt{\kappa}R_0)} \\
&\geq - \frac{(2\alpha+2 + 2\mathcal{H}_{\kappa}(R_0)) \mathcal{H}_{\kappa}(R_0)}{(\sqrt{\kappa}R_0)^{2\alpha+4} \sinh^2 (\sqrt{\kappa}R_0)} (t - 1)
\end{split}
\end{align}
and
\begin{align} \label{eq:arccos3}
\begin{split}
&(\cosh^{-1}(t \cosh (\sqrt{\kappa}R_0)))^{-2\alpha-1} \frac{t \cosh (\sqrt{\kappa}R_0)}{(t^2 \cosh^2 (\sqrt{\kappa}R_0) - 1)^{3/2}} - (\sqrt{\kappa}R_0)^{-2\alpha-1} \frac{\cosh (\sqrt{\kappa}R_0)}{\sinh^3 (\sqrt{\kappa}R_0)} \\
&\geq - \frac{\left( (2\alpha+1) \mathcal{H}_\kappa(R_0) - (\sqrt{\kappa}R_0)^2 + 3\mathcal{H}_\kappa^2(R_0) \right) \mathcal{H}_\kappa(R_0)}{(\sqrt{\kappa}R_0)^{2\alpha+4} \sinh^2 (\sqrt{\kappa}R_0)} (t - 1)
\end{split}
\end{align}
for all $t > 1/\cosh (\sqrt{\kappa}R_0)$.
\end{lemma}

\begin{proof}
Since the function
\begin{equation*}
  f(t) := (\cosh^{-1} (t\cosh(\sqrt{\kappa}R_0))^{-2\alpha}, \quad t > \frac{1}{\cosh (\sqrt{\kappa}R_0)},
\end{equation*}
is convex, \eqref{eq:arccos1} follows from the inequality $f(t) \geq f(1) + f'(1) (t-1)$. The inequalities \eqref{eq:arccos2} and \eqref{eq:arccos3} can be obtained similarly by considering
\begin{align*}
  g(t)&:= \frac{(\cosh^{-1}(t \cosh (\sqrt{\kappa}R_0)))^{-2\alpha-2}}{t^2 \cosh^2 (\sqrt{\kappa}R_0) - 1} \quad \text{and} \\
  h(t)&:= (\cosh^{-1}(t \cosh(\sqrt{\kappa} R_0)))^{-2\alpha-1} \frac{t \cosh(\sqrt{\kappa} R_0)}{(t^2 \cosh^2(\sqrt{\kappa} R_0) - 1)^{3/2}},
\end{align*}
which are also convex functions.
\end{proof}

Using \Cref{lem:arccos}, we first construct a barrier function when $s$ is sufficiently close to 1. Let us denote $\mathcal{K}_{s, \kappa} = \mathcal{K}_{n, s, \kappa}$ in the following lemmas.

\begin{lemma} \label{lem:barrier1}
Let $\delta \in (0,1)$. There are constants $\alpha > 0$ and $s_0 \in (0,1)$, depending only on $n$, $\lambda$, $\Lambda$, $\delta$, and $\sqrt{\kappa} R$, such that the function
\begin{equation*}
v(x) = \max \bigg\lbrace - \left( \frac{\delta}{20} \right)^{-2\alpha}, - \left( \frac{d_{\Hnk}(x, 0)}{5R} \right)^{-2\alpha} \bigg\rbrace
\end{equation*}
is a supersolution to
\begin{equation} \label{eq:Mv}
\frac{(7R)^2}{\mathcal{I}_{0,\kappa}(7R)} \mathcal{M}^+ v(x) + \Lambda \mathcal{H}_\kappa(7R) \leq 0,
\end{equation}
for every $s_0 < s < 1$ and $x \in B_{5R} \setminus \overline{B}_{\delta R/4}$.
\end{lemma}

\begin{proof}
Fix $x$ and let $R_0 := d_{\Hnk}(x,0) \in (\delta R/4, 5R)$. We are going to consider the coordinates centered at $x$. There is an isometry $\varphi \in SO(1,n)$ such that $x = \varphi(0)$ and $0 = \varphi (\frac{1}{\sqrt{\kappa}}\cosh (\sqrt{\kappa}R_0), \frac{1}{\sqrt{\kappa}}\sinh(\sqrt{\kappa} R_0) e_1)$ with $e_1 \in \mathbb{S}^{n-1}$. Notice that 0 denotes $0_{\kappa}=(\frac{1}{\sqrt{\kappa}},0,\cdots,0) \in \Hnk$.

  Let $z \in B_{R_0/2}(x)$, then $z = \varphi(\frac{1}{\sqrt{\kappa}}\cosh(\sqrt{\kappa} r), \frac{1}{\sqrt{\kappa}}\sinh(\sqrt{\kappa} r) \omega)$ for some $r \in [0, R_0/2)$ and $\omega \in \mathbb{S}^{n-1}$. By the hyperbolic law of cosines, we have
\begin{align*}
d_\Hnk(z, 0) 
  &= d_\Hnk\left(\varphi \left( \frac{\cosh(\sqrt{\kappa}r)}{\sqrt{\kappa}}, \frac{\sinh(\sqrt{\kappa}r)}{\sqrt{\kappa}}\omega \right), \varphi\left( \frac{\cosh (\sqrt{\kappa}R_0)}{\sqrt{\kappa}}, \frac{\sinh (\sqrt{\kappa}R_0)}{\sqrt{\kappa}} e_1 \right) \right) \\
  &= d_{\Hnk} \left( \left( \frac{\cosh(\sqrt{\kappa} r)}{\sqrt{\kappa}}, \frac{\sinh(\sqrt{\kappa}r)}{\sqrt{\kappa}} \right), \left( \frac{\cosh(\sqrt{\kappa} R_{0})}{\sqrt{\kappa}}, \frac{\sinh(\sqrt{\kappa}R_{0})}{\sqrt{\kappa}} \right) \right) \\
  &= \frac{1}{\sqrt{\kappa}}\cosh^{-1} (A-B)
\end{align*}
where $A = \cosh ( \sqrt{\kappa} r) \cosh (\sqrt{\kappa}R_0)$ and $B = \sinh (\sqrt{\kappa}r) \sinh (\sqrt{\kappa}R_0) \omega_1$. Similarly, we have
\begin{equation*}
  d_\Hnk(\exp_x(-\exp_x^{-1}z), 0) = \frac{1}{\sqrt{\kappa}}\cosh^{-1} (A+B).
\end{equation*}
Thus, we obtain
\begin{equation*}
  \delta(v, x, z) = - (5\sqrt{\kappa}R)^{2\alpha} \frac{(\cosh^{-1}(A-B))^{-2\alpha} + (\cosh^{-1}(A+B))^{-2\alpha} - 2(\sqrt{\kappa}R_0)^{-2\alpha}}{2}.
\end{equation*}
Since $(\cosh^{-1}(\cdot))^{-2\alpha}$ is convex at $A$, we obtain
\begin{align*}
\delta(v, x, z)
  \leq& -(5\sqrt{\kappa}R)^{2\alpha} \left( \alpha(2\alpha+1) \frac{(\cosh^{-1}A)^{-2\alpha-2}}{(A^2-1)^{1/2}} + \alpha \frac{A(\cosh^{-1}A)^{-2\alpha-1}}{(A^2-1)^{3/2}} \right)B^2 \\
      &- (5\sqrt{\kappa}R)^{2\alpha} \left( (\cosh^{-1}A)^{-2\alpha} - (\sqrt{\kappa}R_0)^{-2\alpha} \right).
\end{align*}
Moreover, by applying \Cref{lem:arccos} with $t=\cosh (\sqrt{\kappa} r)$, we have
\begin{equation} \label{eq:second_difference}
\begin{split}
\delta (v, x, z)
  \leq& ~\alpha(2\alpha+1) c_\delta \left( (2\alpha+2+2\mathcal{H}_{\kappa}(R_0)) \mathcal{H}_{\kappa}(R_0) \frac{\cosh (\sqrt{\kappa}r)-1}{(\sqrt{\kappa}R_0)^2} - 1 \right) \frac{\sinh^2 (\sqrt{\kappa}r)}{(\sqrt{\kappa}R_0)^2} \omega_1^2 \\
      &+ \alpha c_\delta \left( \left( (2\alpha+1) \mathcal{H}_\kappa(R_0) - (\sqrt{\kappa}R_0)^2 + 3\mathcal{H}_\kappa^2(R_0) \right) \frac{\cosh (\sqrt{\kappa}r) - 1}{(\sqrt{\kappa}R_0)^2} - 1 \right) \\
      &\qquad \times \mathcal{H}_\kappa(R_0) \frac{\sinh^2 (\sqrt{\kappa}r)}{(\sqrt{\kappa}R_0)^2} \omega_1^2 + 2\alpha c_\delta \mathcal{H}_\kappa(R_0) \frac{\cosh (\sqrt{\kappa}r) - 1}{(\sqrt{\kappa}R_0)^2},
\end{split}
\end{equation}
where $c_\delta = (20/\delta)^{2\alpha}$.
Let us now compute
\begin{equation} \label{eq:I}
\begin{split}
\mathcal{M}^+ v (x)
&\leq \int_{B(x, \frac{R_0}{2})} \left( \Lambda \delta^+(v, x, z) - \lambda \delta^-(v, x, z) \right) \mathcal{K}_{s, \kappa}(d_{\Hnk}(z, x)) \, \d \mu_{\Hnk}(z) \\
&\quad+ \int_{\Hn \setminus B(x, \frac{R_0}{2})} \left( \Lambda \delta^+(v, x, z) - \lambda \delta^-(v, x, z) \right) \mathcal{K}_{s, \kappa}(d_{\Hnk}(z, x)) \, \d \mu_{\Hnk}(z) \\
&=: I_1 + I_2.
\end{split}
\end{equation}
We take $\alpha = \alpha(n,\lambda, \Lambda, \sqrt{\kappa}R) > 0$ sufficiently large so that
\begin{equation} \label{eq:alpha}
\lambda(2\alpha+1) \fint_{\mathbb{S}^{n-1}} \omega_1^2 \, \d \sigma - \Lambda \mathcal{H}_{\kappa}(7R) > C_1\Lambda \mathcal{H}_{\kappa}(7R),
\end{equation}
for some universal constant $C_1 > 0$ to be determined later. Then, by \eqref{eq:second_difference} we have
\begin{equation} \label{eq:I_1}
\begin{split}
I_1
  &\leq \Lambda \alpha c_\delta \left( (2\alpha+1)\mathcal{H}_\kappa(R_0) -(\sqrt{\kappa} R_0)^2 + 3\mathcal{H}_\kappa^2(R_0) \right) \mathcal{H}_\kappa(R_0) \\
  &\qquad\qquad\qquad \times \int_{B_{R_{0}/2}} \frac{\cosh (\sqrt{\kappa}r) - 1}{(\sqrt{\kappa}R_0)^2} \frac{\sinh^2 (\sqrt{\kappa}r)}{(\sqrt{\kappa}R_0)^2} \omega_1^2 \mathcal{K}_{s, \kappa}(d_{\Hnk}(z, x)) \, \d \mu_{\Hnk}(z) \\
      &\quad + \Lambda \alpha(2\alpha+1) c_\delta \left( 2\alpha+2+2\mathcal{H}_\kappa(R_0) \right) \mathcal{H}_\kappa(R_0) \\
      &\qquad\qquad\qquad \times \int_{B_{R_0/2}} \frac{\cosh (\sqrt{\kappa}r) - 1}{(\sqrt{\kappa}R_0)^2} \frac{\sinh^2 (\sqrt{\kappa}r)}{(\sqrt{\kappa}R_0)^2} \omega_1^2 \mathcal{K}_{s, \kappa}(d_{\Hnk}(z, x)) \, \d \mu_{\Hnk}(z) \\
      &\quad + \alpha c_\delta \int_{B_{\frac{R_0}{2}}} \left( 2\Lambda \mathcal{H}_\kappa(R_0) \frac{\cosh (\sqrt{\kappa}r) - 1}{(\sqrt{\kappa}R_0)^2} - \lambda (2\alpha+1+\mathcal{H}_\kappa(R_0)) \frac{\sinh^2(\sqrt{\kappa} r)}{(\sqrt{\kappa}R_0)^2} \omega_1^2 \right) \\
      &\qquad\qquad\qquad \times \mathcal{K}_{s, \kappa} (d_{\Hnk}(z, x))\, \d \mu_{\Hnk}(z) \\
=&~ \alpha c_\delta \left( I_{1,1} + I_{1,2} + I_{1,3} \right).
\end{split}
\end{equation}
We use \eqref{eq:alpha} to estimate $I_{1,3}$ as follows:
\begin{equation}\label{eq:I_11}
\begin{split}
I_{1,3}
  &= \int_{B_{R_0/2}} \left( 4\Lambda \mathcal{H}_\kappa(R_0) - 4\lambda (2\alpha+1+\mathcal{H}_\kappa(R_0)) \cosh^2 \left( \frac{\sqrt{\kappa}}{2} \right)\omega_1^2 \right) \frac{\sinh^2(\sqrt{\kappa}\frac{r}{2})}{(\sqrt{\kappa}R_0)^2} \mathcal{K}_{s, \kappa} \, \d \mu_{\Hnk}(z) \\
  &\leq \int_0^{R_0/2} \left( 4\Lambda| \mathbb{S}^{n-1} |\mathcal{H}_\kappa(7R) - 4\lambda(2\alpha+1) \int_{\mathbb{S}^{n-1}} \omega_1^2 \, \d \sigma \right) \frac{\sinh^2(\sqrt{\kappa} \frac{r}{2})}{(\sqrt{\kappa}R_0)^2} \mathcal{K}_{s, \kappa} \frac{\sinh^{n-1} (\sqrt{\kappa}r)}{\sqrt{\kappa}^{n-1}} \, \d r \\
  &\leq -4C_1\Lambda \frac{\mathcal{H}_\kappa(7R)}{(\sqrt{\kappa}R_0)^2} \int_0^{R_0/2} | \mathbb{S}^{n-1} | \left( \frac{r}{2} \right)^2 \mathcal{K}_{s, \kappa}(r) \frac{\sinh^{n-1} (\sqrt{\kappa}r)}{\sqrt{\kappa}^{n-1}} \, \d r \\
&= -C_1\Lambda \mathcal{H}_\kappa(7R) \frac{\mathcal{I}_{0,\kappa}(R_0/2)}{(\sqrt{\kappa}R_0)^2}.
\end{split}
\end{equation}
For $I_{1,1}$ and $I_{1,2}$, we observe that $\cosh (\sqrt{\kappa}r)-1 \leq C \kappa  r^2$ and $\sinh^2 (\sqrt{\kappa}r) \leq C\kappa r^2$ for $r \in [0, R_0/2]$, where $C$ is some constant depending on $\sqrt{\kappa}R$.
Thus, by using \Cref{lem:kernel-asymp} and \eqref{eq:scaling-kernel} we obtain
\begin{align}\label{eq:I_23}
\begin{split}
I_{1,1} + I_{1,2} 
&\leq C\Lambda \int_0^{R_0/2} \int_{\mathbb{S}^{n-1}} \frac{\cosh (\sqrt{\kappa} r)-1}{\kappa R_0^2} \frac{\sinh^2 (\sqrt{\kappa} r)}{\kappa R_0^2} \omega_1^2 \mathcal{K}_{s, \kappa}(r) \frac{\sinh^{n-1}(\sqrt{\kappa} r)}{\sqrt{\kappa}^{n-1}} \, \d\sigma \, \d r \\
&\leq C\Lambda(1-s) \frac{\sqrt{\kappa}^{1+s}}{R_{0}^{4}} \int_0^{R_0/2} r^{4-s} I_{\frac{n}{2}-1}\left( \frac{n-1}{2}\sqrt{\kappa}r \right)K_{\frac{n}{2}+s} \left( \frac{n-1}{2}\sqrt{\kappa} r\right) \, \d r \\
&\leq C\Lambda(1-s) \frac{\sqrt{\kappa}^{-4+2s}}{R_0^{4}} \int_0^{\frac{n-1}{4}\sqrt{\kappa}R_{0}} r^{4-s} I_{\frac{n}{2}-1}(r) K_{\frac{n}{2}+s} (r) \, \d r \\
& \leq C\Lambda(1-s) \frac{\sqrt{\kappa}^{-4+2s}}{R_0^4} A^{4-s}_{\frac{n}{2}+s, \frac{n}{2}-1}\left( \frac{n-1}{4} \sqrt{\kappa}R_0 \right),
\end{split}
\end{align}
where $A$ is the function defined in \eqref{eq:A-definite}.

On the other hand, by using the fact that $v$ is bounded and \Cref{prop:I_infty_I_0}, we obtain
\begin{equation} \label{eq:I_2}
I_2 \leq C\Lambda \frac{\mathcal{I}_{\infty,\kappa}(R_0/2)}{(R_0/2)^2} \leq C\Lambda \frac{1-s}{s} \mathcal{H}_{\kappa}(7R) \frac{\mathcal{I}_{0,\kappa} (7R)}{(7R)^2}.
\end{equation}
Thus, \eqref{eq:I}, \eqref{eq:I_1}, \eqref{eq:I_11}, \eqref{eq:I_23}, \eqref{eq:I_2}, and \Cref{lem:decreasing} yield
\begin{equation} \label{eq:Mv1}
\begin{split}
&\frac{(7R)^2}{\mathcal{I}_{0,\kappa} (7R)} \mathcal{M}^+ v (x) \\
&\leq C \alpha\Lambda \left( -C_1 + C(1-s) \frac{(7R)^2}{\mathcal{I}_{0,\kappa} (7R)} \frac{\sqrt{\kappa}^{-4+2s}}{R_0^4} A^{4-s}_{\frac{n}{2}+s, \frac{n}{2}-1}\left(\frac{n-1}{4}\sqrt{\kappa}R_0\right) + \frac{1-s}{s} \right) \mathcal{H}_{\kappa} (7R).
\end{split}
\end{equation}
Recall from \Cref{prop:limit_zero} that $\mathcal{I}_0(7R) \to C$ as $s \to 1$. Moreover, the function $A^{4-s}_{\frac{n}{2}+s, \frac{n}{2}-1}$ does not blow up as $s \to 1$ by \Cref{lem:A}. Thus, the second and the third terms in \eqref{eq:Mv1} can be made as small as we want by choosing $s_{0}$ close to 1. Therefore, the proof is finished by assuming that we have taken $\alpha$ sufficiently large so that \eqref{eq:Mv} holds.
\end{proof}

In the following lemma, we construct a barrier function for any $s \in (s_0, 1)$ for given $s_0 \in (0,1)$.

\begin{lemma} \label{lem:barrier2}
Given $s_0 \in (0,1)$ and $\delta \in (0,1)$, there exist universal constants $\alpha > 0$ and $\eta \in (0, 1/4]$, depending only on $n$, $\lambda$, $\Lambda$, $\delta$, $\sqrt{\kappa}R$, and $s_0$, such that the function
\begin{equation*}
v(x) = \max \bigg\lbrace - \left( \frac{\eta \delta}{20} \right)^{-2\alpha}, - \left( \frac{d_\Hnk(x,0)}{5R} \right)^{-2\alpha} \bigg\rbrace
\end{equation*}
is a supersolution to
\begin{equation} \label{eq:outside}
\frac{(7R)^2}{\mathcal{I}_{0,\kappa} (7R)} \mathcal{M}^+ v(x) + \Lambda \mathcal{H}_{\kappa} (7R) \leq 0,
\end{equation}
for every $s_0 < s < 1$ and $x \in B_{5R} \setminus \overline{B}_{\delta R/4}$.
\end{lemma}

\begin{proof}
Let $s_1$ and $\alpha_1$ be the $s_0$ and $\alpha$ in \Cref{lem:barrier1}, respectively. When $s \in [s_1, 1)$, the desired result holds with $\alpha_1$ and $\eta = 1/4$.

Let us now assume $s \in (s_0, s_1)$. For $x$ with $R_0 := d_\Hnk(x, 0) \in (\delta R/4, 5R)$, we know that $v \in C^2(B(x, R_0/2))$ and that $\delta^+(v, x, z)$ is bounded for $z \in \Hnk \setminus B(x, R_0/2)$. Thus, we have
\begin{equation} \label{eq:not_integrable}
\frac{(7R)^2}{\mathcal{I}_{0,\kappa} (7R)} \mathcal{M}^+ v (x) \leq C \mathcal{H}_{\kappa} (7R) - \lambda \int_{\Hnk} \delta^-(v, x, z) \mathcal{K}_{s, \kappa}(d_{\Hnk}(z, x)) \, \d \mu_{\Hnk}(z).
\end{equation}
If we take $\alpha = \max \lbrace \alpha_1, n/2 \rbrace$, then the function $\delta^-(-(d_\Hnk(\cdot,0)/5R)^{-2\alpha}, x, z)$ is not integrable. Therefore, the last integral in \eqref{eq:not_integrable} can be made arbitrarily large, by taking $\eta$ small. In particular, we choose $\eta$ so that \eqref{eq:outside} holds.
\end{proof}

\begin{corollary} \label{cor:barrier}
Let $\delta \in (0,1)$ and assume $0 < s_0 \leq s < 1$. Then, there is a function $v_\delta$ such that
\begin{equation*}
\begin{cases}
v_\delta \geq 0 &\text{in}~ \Hnk \setminus B_{5R}, \\
v_\delta \leq 0 &\text{in}~ B_{2R}, \\
\frac{(7R)^2}{\mathcal{I}_{0,\kappa} (7R)} \mathcal{M}^+ v_\delta + \Lambda \mathcal{H}_{\kappa} (7R) \leq 0 &\text{in}~ B_{5R} \setminus \overline{B}_{\delta R/4}, \\
\frac{(7R)^2}{\mathcal{I}_{0,\kappa} (7R)} \mathcal{M}^+ v_\delta \leq C\Lambda \mathcal{H}_{\kappa} (7R) &\text{in}~ B_{5R}, \\
v \geq -C &\text{in}~ B_{5R}, \\
\end{cases}
\end{equation*}
for some universal constant $C > 0$, depending only on $n$, $\lambda$, $\Lambda$, $\delta$, $\sqrt{\kappa}R$, and $s_0$.
\end{corollary}

\begin{proof}
Let $\alpha$ and $\eta$ be the constants given in \Cref{lem:barrier2}, and define a function $v_\delta(x) = \psi(d_\Hnk^2(x,0) / R^2)$, where $\psi$ is a smooth and increasing function on $[0, \infty)$ such that
\begin{equation*}
\psi(t) = \left( \frac{3^2}{5^2} \right)^{-\alpha} - \left( \frac{t}{5^2} \right)^{-\alpha} \quad\text{if}~ t \geq (\eta \delta)^2.
\end{equation*}
We already know from \Cref{lem:barrier2} that $\frac{(7R)^2}{\mathcal{I}_{0,\kappa} (7R)} \mathcal{M}^+ v_\delta + \Lambda \mathcal{H}_{\kappa} (7R) \leq 0$ in $B_{5R} \setminus \overline{B}_{\delta R/4}$. Finally, for $x \in \overline{B}_{\delta R/4}$, we have $|\delta(v_\delta, x, z)| \leq C \mathcal{H}_{\kappa} (7R) d_\Hnk(x, z)^2/R^2$ for $z \in B_R(x)$ and $|\delta (v_\delta,x,z)| \leq C$ for $z \in \Hnk \setminus B_R(x)$ with a uniform constant $C > 0$. Therefore, we conclude $\frac{(7R)^2}{\mathcal{I}_{0,\kappa} (7R)} \mathcal{M}^+ v_\delta \leq C\Lambda \mathcal{H}_{\kappa} (7R)$ in $B_{5R}$, with the help of \Cref{prop:I_infty_I_0}.
\end{proof}

\section{\texorpdfstring{$L^\varepsilon$}{L}-estimate} \label{sec:L_eps_estimate}

In this section, we prove the so-called $L^\varepsilon$-estimate, which connects a pointwise estimate to an estimate in measure. Such a result forms a basis for the proofs of the Harnack inequality and H\"older estimate. From now on, we will prove the results only on $\Hn$ since \Cref{thm:Harnack} and \Cref{thm:Holder} can be derived from the results on $\Hn$ by using a simple scaling argument. Moreover, since the essential results in the previous sections have been proved on $\Hnk$, one may easily reprove forthcoming results on $\Hnk$. We write $\mathcal{K}_{s}=\mathcal{K}_{n, s, 1}$, $\mathcal{H}=\mathcal{H}_{1}$, $\mathcal{S}=\mathcal{S}_{1}$, and $\mathcal{T}=\mathcal{T}_{1}$ for simplicity in the sequel.

\begin{lemma} \label{lem:base}
Assume $0 < s_0 \leq s < 1$, and let $\delta \in (0,1)$. If $u \in C^2(B_{7R})$ is a nonnegative function on $\Hn$ satisfying $\frac{(7R)^2}{\mathcal{I}_0(7R)} \mathcal{M}^- u \leq \varepsilon_\delta$ in $B_{7R}$ and $\inf_{B_{2R}} u \leq 1$, then
\begin{equation*}
\frac{|\lbrace u \leq M_\delta \rbrace \cap B_{\delta R}|}{|B_{7R}|} \geq \mu_\delta,
\end{equation*}
where $\varepsilon_\delta > 0$, $\mu_\delta \in (0,1)$, and $M_\delta > 1$ are universal constants depending only on $n$, $\lambda$, $\Lambda$, $\delta$, $R$ and $s_0$.
\end{lemma}

\begin{proof}
Let $v_\delta$ be the barrier function constructed in \Cref{cor:barrier} and define $w = u + v_\delta$. Then $w$ satisfies $w \geq 0$ in $\Hn \setminus B_{5R}$, $\inf_{B_{2R}} w \leq 1$, and $\mathcal{M}^-w \leq \frac{\mathcal{I}_0(7R)}{(7R)^2} \varepsilon_\delta + \mathcal{M}^+ v_\delta$ in $B_{5R}$. By applying \Cref{thm:ABP} to $w$ with its envelope $\Gamma_w$, we have
\begin{equation*}
|B_R| \leq \sum_j c F^n |Q^j_\alpha|,
\end{equation*}
where
\begin{equation*}
F = \mathcal{S}(7R) \bigg( \Lambda \mathcal{H}(7R) + \frac{R^2}{\mathcal{I}_0(R)} \bigg( \frac{\mathcal{I}_0(7R)}{(7R)^2}\varepsilon_\delta + \max_{\overline{Q}^j_\alpha} \mathcal{M}^+ v_\delta \bigg) \bigg)_+
\end{equation*}
and $c = C \cosh^{n-1} (C \mathcal{T}^{2}(r_0) r_0 F)  (C\mathcal{T}^{2}(r_0) F)^{(n-1)\log \cosh(C\mathcal{T}^{2}(r_0) r_0 F)} \mathcal{T}^{2n}(r_0)$. We obtain by \Cref{prop:monotonicity}
\begin{equation*}
F \leq \mathcal{S}(7R) \bigg( \varepsilon_\delta + \Lambda \mathcal{H}(7R) + \frac{(7R)^2}{\mathcal{I}_0(7R)} \max_{\overline{Q}^j_\alpha} \mathcal{M}^+ v_\delta \bigg)_+.
\end{equation*}
Since $\Lambda \mathcal{H}(7R) + \frac{(7R)^2}{\mathcal{I}_0(7R)}\mathcal{M}^+ v_\delta \leq 0$ in $B_{5R} \setminus \overline{B}_{\delta R/4}$ and $\frac{(7R)^2}{\mathcal{I}_0(7R)} \mathcal{M}^+ v_\delta \leq C\Lambda \mathcal{H}(7R)$ in $B_{5R}$, we have
\begin{equation*}
|B_R| \leq C \varepsilon_\delta^n \sum_{\overline{Q}_\alpha^j \cap \overline{B}_{\delta R/4} = \emptyset} |Q_\alpha^j| + C \sum_{\overline{Q}_\alpha^j \cap \overline{B}_{\delta R/4} \neq \emptyset} |Q_\alpha^j|
\end{equation*}
for some universal constant $C > 0$, depending on $R$. By taking $\varepsilon_\delta > 0$ sufficiently small, we have
\begin{equation*}
|B_{7R}| \leq C \sum_{\overline{Q}_\alpha^j \cap \overline{B}_{\delta R/4} \neq \emptyset} |Q_\alpha^j|.
\end{equation*}
By using \Cref{lem:ABP_discrete} (iv), we obtain
\begin{equation*}
|B_{7R}| \leq C \sum_{\overline{Q}_\alpha^j \cap \overline{B}_{\delta R/4} \neq \emptyset} |B(z_\alpha^j, 2r_0) \cap \lbrace w \leq \Gamma_w + C \rbrace|.
\end{equation*}
Whenever $\overline{Q}_\alpha^j \cap \overline{B}_{\delta R/4} \neq \emptyset$, the ball $B(z_\alpha^j, 2r_0)$ is contained in $B_{\delta R}$ if we have taken $\rho_0 = \delta/4$. Indeed, for $z \in B(z_\alpha^j, 2r_0)$
\begin{equation*}
d_\Hn(z,0) \leq d_\Hn(z, z_\alpha^j) + d_\Hn(z_\alpha^j, z_\ast) + d_\Hn(z_\ast, 0) \leq 2r_0 + r_0 + \delta R/4 < \delta R,
\end{equation*}
where $z_\ast$ is a point in $\overline{Q}_\alpha^j \cap \overline{B}_{\delta R}$. By taking a subcover of $\lbrace B(z_\alpha^j, 2r_0) \rbrace$ with finite overlapping and using $v_\delta \geq -C$ in $B_{5R}$, we arrive at
\begin{equation*}
|B_{7R}| \leq C |\lbrace u \leq M_\delta \rbrace \cap B_{\delta R}|
\end{equation*}
for some $M_\delta > 1$. Taking $\mu_\delta = 1/C$ finishes the proof.
\end{proof}

\Cref{lem:base}, together with the Calder\'on–Zygmund technique developed in \cite{Cab97}, provides the following $L^\varepsilon$-estimate. As in \cite{Cab97}, we fix $\delta = \frac{2c_1}{c_2} \delta_0$ and $\delta_1 = \delta_0(1-\delta_0)/2 \in (0,1)$. Let $k_R$ be the integer satisfying
\begin{equation*}
c_2 \delta_0^{k_R - 1} < R \leq c_2 \delta_0^{k_R - 2},
\end{equation*}
which is the generation of a dyadic cube whose size is comparable to that of some ball of radius $R$. 

\begin{lemma} \label{lem:L-eps}
Assume $0 < s_0 \leq s < 1$. Let $\varepsilon_\delta$, $\mu_\delta$, and $M_\delta$ be the constants in \Cref{lem:base}. Let $u \in C^2(B_{7R})$ be a nonnegative function on $\Hn$ satisfying $\frac{(7R)^2}{\mathcal{I}_0(7R)} \mathcal{M}^- u \leq \varepsilon_\delta$ in $B_{7R}$ and $\inf_{B_{\delta_1 R}} u \leq 1$. If $Q_1$ is a dyadic cube of generation $k_R$ such that $\inf_{x \in Q_1} d_\Hn(x, 0) \leq \delta_1 R$, then
\begin{equation*}
|\lbrace u > M_\delta^i \rbrace \cap Q_1| \leq (1 - c_\delta)^i |Q_1|.
\end{equation*}
for all $i= 1, 2, \cdots$. As a consequence, we have
\begin{equation*}
|\lbrace u > t \rbrace \cap Q_1| \leq C t^{-\varepsilon} |Q_1|, \quad t > 0,
\end{equation*}
for some universal constants $C > 0$ and $\varepsilon > 0$.
\end{lemma}

\begin{corollary} [Weak Harnack inequality] \label{cor:WHI}
Assume $0 < s_0 \leq s < 1$. If $u \in C^2(B_{2R})$ is a nonnegative function satisfying $\mathcal{M}^- u \leq C_0$ in $B_{2R}$, then
\begin{equation*}
\left( \fint_{B_R} u^p \, \d \mu_{\Hn} \right)^{1/p} \leq C \left( \inf_{B_R} u + C_0 \frac{R^2}{\mathcal{I}_0(R)} \right),
\end{equation*}
where $p > 0$ and $C > 0$ are universal constants depending only on $n$, $\lambda$, $\Lambda$, $R$, and $s_0$.
\end{corollary}

See, e.g., \cite[Theorem 8.1]{Cab97} for the proof of \Cref{cor:WHI}.

\section{Harnack inequality} \label{sec:Harnack}

The purpose of this section is to prove the Krylov–Safonov Harnack inequality by using \Cref{lem:L-eps}. A simple scaling argument will provide \Cref{thm:Harnack}.

\begin{theorem}
Assume $0 < s_0 \leq s < 1$. If a nonnegative function $u \in C^2(B_{7R})$ satisfies
\begin{equation*}
\frac{(7R)^2}{\mathcal{I}_0(7R)} \mathcal{M}^- u \leq \varepsilon_0 \quad\text{and}\quad \frac{(7R)^2}{\mathcal{I}_0(7R)} \mathcal{M}^+ u \geq -\varepsilon_0 \quad\text{in}~ B_{7R}
\end{equation*}
and $\inf_{B_{\delta_1 R}} u \leq 1$, then
\begin{equation*}
\sup_{B_{\delta_1 R/4}} u \leq C,
\end{equation*}
where $\varepsilon_0 > 0$ and $C > 0$ are universal constants depending only on $n$, $\lambda$, $\Lambda$, $R$, and $s_0$.
\end{theorem}

\begin{proof}
Let $\varepsilon$ and $\varepsilon_\delta$ be the constants given in \Cref{lem:L-eps}, and let $t > 0$ be the minimal value such that the following holds:
\begin{equation*}
u(x) \leq h_t(x) := t \bigg( \frac{1}{\mathcal{D}} \bigg( 1 - \frac{d_\Hn(x,z_0)}{\delta_1R} \bigg)^{\log_2\mathcal{D}}\bigg)^{-1/\varepsilon} \quad \text{for all} ~ x \in B_{\delta_1 R},
\end{equation*}
where for $\mathcal{D} = 2^n \cosh^{n-1} (2\delta_0 R)$. Since $\sup_{B_{\delta_1 R/4}} u \leq t \mathcal{D}^{-\frac{1}{\varepsilon} \log_2(3/8)}$, we can conclude the theorem once we show that $t \leq C$ for some universal constant $C$.

Let $x_0 \in B_{\delta_1 R}$ be a point such that $u(x_0) = h_t(x_0)$. Let $d = \delta_1 R - d_{\Hn}(x_0, 0)$, $r = d/2$, and $A = \lbrace u > u(x_0)/2 \rbrace$, then we have
\begin{equation*}
u(x_0) = h_t(x_0) = t \mathcal{D}^{1/\varepsilon} \left( \frac{2r}{\delta_1 R} \right)^{-\frac{1}{\varepsilon} \log_2 \mathcal{D}}.
\end{equation*}
We apply \Cref{lem:L-eps} to $u$ to obtain
\begin{equation*}
|A \cap Q_1| \leq C \left( \frac{u(x_0)}{2} \right)^{-\varepsilon} |Q_1| \leq C t^{-\varepsilon} \frac{1}{\mathcal{D}} \left( \frac{r}{R} \right)^{\log_2 \mathcal{D}} |Q_1|,
\end{equation*}
where $Q_1$ is the unique dyadic cube of generation $k_R$ that contains the point $x_0$.

We will show that there is a small constant $\theta > 0$ such that
\begin{equation} \label{eq:ball_theta}
|A^c \cap Q_2| \leq \frac{1}{2} |Q_2|,
\end{equation}
where $Q_2 \subset Q_1$ is the dyadic cube of generation $k_{\theta r/14}$ containing the point $x_0$, provided that $t$ is large. However, when $t$ is sufficiently large, we also have
\begin{equation*}
|A \cap Q_2| \leq |A \cap Q_1| \leq \frac{C}{t^\varepsilon \mathcal{D}} \left( \frac{r}{R} \right)^{\log_2 \mathcal{D}} | B(z, c_2 \delta_0^{k_R}) | \leq \frac{C}{t^\varepsilon} | B(z, c_1 \delta_0^{k_{r\theta/14}}) | \leq \frac{C}{t^\varepsilon} | Q_2 | < \frac{1}{2}|Q_2|,
\end{equation*}
where $B(z, c_1 \delta_0^{k_{\theta r/14}})$ is a ball contained in $Q_2$. This contradicts to \eqref{eq:ball_theta} and will lead us to a conclusion that $t$ is uniformly bounded.

Let us now focus on proving \eqref{eq:ball_theta}. For every $x \in B(x_0, \theta r)$, we have
\begin{equation*}
u(x) \leq h_t(x) \leq t \left( \frac{1}{\mathcal{D}} \left( \frac{d-\theta r}{\delta_1 R} \right)^{\log_2\mathcal{D}}\right)^{-1/\varepsilon} = \left( 1- \frac{\theta}{2} \right)^{-\frac{1}{\varepsilon} \log_2 \mathcal{D}} u(x_0).
\end{equation*}
We define a function
\begin{equation*}
v(x) := \left( 1- \frac{\theta}{2} \right)^{-\frac{1}{\varepsilon} \log_2 \mathcal{D}} u(x_0) - u(x).
\end{equation*}
Since we will apply \Cref{lem:L-eps}, we need a function which is nonnegative on the whole space. Thus, we apply \Cref{lem:L-eps} to $w:= v^+$ in $B(x_0, 7(\theta r/14))$. For $x \in B(x_0, 7(\theta r/14))$, we have
\begin{align} \label{eq:Mw}
\begin{split}
\mathcal{M}^- w(x)
&\leq \mathcal{M}^- v(x) + \mathcal{M}^+ v_-(x) \\
&\leq - \mathcal{M}^+ u(x) + \Lambda \int_{\Hn \setminus B(x_0, \theta r)} v^-(z) \mathcal{K}_{s} (d_\Hn(z, x)) \, \d \mu_{\Hn}(z) \\
&\leq \frac{\mathcal{I}_0(7R)}{(7R)^2} \varepsilon_0 + \Lambda \int_{\Hn \setminus B(x_0, \theta r)} (u(z) - (1-\theta/2)^{-\frac{\log_2 \mathcal{D}}{\varepsilon} } u(x_0))^+ \mathcal{K}_s(d_{\Hn}(z, x)) \, \d \mu_{\Hn}(z).
\end{split}
\end{align}
To compute the last integral in \eqref{eq:Mw}, we introduce another auxiliary function
\begin{equation*}
g_\beta(x) := \beta \left( 1- \frac{d_{\Hn}(x, 0)^2}{R^2} \right)^+,
\end{equation*}
with the largest number $\beta > 0$ satisfying $u \geq g_\beta$. From the assumption $\inf_{B_{\delta_1 R}} u \leq 1$, we have $(1-\delta_1^2) \beta \leq 1$. Let $x_1 \in B_R$ be a point where $u(x_1) = g_\beta(x_1)$. Since
\begin{align*}
\int_\Hn \delta^-(u, x_1, z) \mathcal{K}_s(d_\Hn(z, x_1)) \, \d \mu_{\Hn}(z) 
&\leq \int_\Hn \delta^-(g_\beta, x_1, z) \mathcal{K}_s(d_\Hn(z, x_1)) \, \d \mu_{\Hn}(z) \\
&\leq C \mathcal{H}(7R) \frac{\mathcal{I}_0(7R)}{(7R)^2},
\end{align*}
we obtain that
\begin{align*}
\varepsilon_0
&\geq \frac{(7R)^2}{\mathcal{I}_0(7R)} \mathcal{M}^- u(x_1) \\
&\geq \lambda \frac{(7R)^2}{\mathcal{I}_0(7R)} \int_\Hn \delta^+(u,x_1, z) \mathcal{K}_s(d_\Hn(z, x_1)) \, \d \mu_{\Hn}(z) - C\Lambda \mathcal{H}(7R).
\end{align*}
Thus, we have
\begin{align} \label{eq:u-c}
\begin{split}
&\int_\Hn (u(z) - c)^+ \mathcal{K}_s(d_\Hn(z, x_1)) \, \d \mu_{\Hn}(z) \\
&\leq \int_\Hn \delta^+(u,x_1, z) \mathcal{K}_s(d_\Hn(z, x_1)) \, \d \mu_{\Hn}(z) \leq C \mathcal{H}(7R) \frac{\mathcal{I}_0(7R)}{(7R)^2},
\end{split}
\end{align}
where $c:= 1/(1-\delta_1^2)$.

If $u(x_0) \leq c$, then we find an upper bound $t = u(x_0) (\delta_1 R / d)^{-\frac{1}{\varepsilon} \log_2 \mathcal{D}} \leq c \delta_1^{-\frac{1}{\varepsilon} \log_2 \mathcal{D}}$, which finishes the proof. Otherwise, it follows from \eqref{eq:Mw} and \eqref{eq:u-c} that
\begin{align*}
\mathcal{M}^- w(x)
&\leq \frac{\mathcal{I}_0(7R)}{(7R)^2} \varepsilon_0 + \Lambda \int_{\Hn \setminus B(x_0, \theta r)} (u(z) - c)_+ \mathcal{K}_s(d_\Hn(z, x)) \, \d \mu_{\Hn}(z) \\
&\leq \frac{\mathcal{I}_0(7R)}{(7R)^2} \varepsilon_0 + \Lambda M \int_{\Hn \setminus B(x_0, \theta r)} (u(z) - c)_+ \mathcal{K}_s(d_\Hn(z, x_1)) \, \d \mu_{\Hn}(z) \\
&\leq \frac{\mathcal{I}_0(7R)}{(7R)^2} \varepsilon_0 + C M \mathcal{H}(7R)\frac{\mathcal{I}_0(7R)}{(7R)^2},
\end{align*}
where
\begin{equation*}
M = \sup \left\lbrace \frac{\mathcal{K}_s(d_\Hn(z, x))}{\mathcal{K}_s(d_\Hn(z, x_1))}: x \in B(x_0, \theta r/2), x_1 \in B_R, z \in \Hn \setminus B(x_0, \theta r) \right\rbrace.
\end{equation*}
Let $d=d_\Hn(z,x)$ and $d_1=d_\Hn(z,x_1)$ for the sake of brevity. We recall from \Cref{lem:kernel-asymp} that the kernel $\mathcal{K}_{s}$ is comparable with the function
\begin{align*}
R^{s-\frac{1}{2}} \sinh^{-\frac{n-1}{2}}(R) K_{\frac{n}{2} + s} \left(\frac{n-1}{2}R \right).
\end{align*}
If $d \geq d_1$, then by \cite[Chapter 4]{Laf91} we have
\begin{equation*}
  \frac{\mathcal{K}_s(d)}{\mathcal{K}_s(d_1)} \leq C \left( \frac{d_1}{d} \right)^{1/2+s} \left( \frac{\sinh d_1}{\sinh d} \right)^{\frac{n-1}{2}} \frac{K_{n/2+s}(\frac{n-1}{2}d)}{K_{n/2+s}(\frac{n-1}{2}d_1)} \leq C \left( \frac{\sinh d_1}{\sinh d} \right)^{\frac{n-1}{2}} e^{\frac{n-1}{2}(d_1-d)} \leq C.
\end{equation*}
If $d < d_1$, by \cite[Theorem 3.1]{Laf91} we have
\begin{equation*}
  \frac{\mathcal{K}_s(d)}{\mathcal{K}_s(d_1)} \leq C
  \left(\frac{d_1}{d} \right)^{\frac{n+1}{2} + 2s } \left( \frac{\sinh d_1}{\sinh d} \right)^{\frac{n-1}{2}} e^{\frac{n-1}{2}(d_1-d)} < \left(\frac{d_1}{d} \right)^{\frac{n+5}{2}} \left( \frac{\sinh d_1}{\sinh d} \right)^{\frac{n-1}{2}} e^{\frac{n-1}{2}(d_1-d)}.
\end{equation*}
Since $d_1-d \le d(x,x')$, we can bound $\mathcal{K}_s(d)/\mathcal{K}_s(d_1)$ by a constant depending on $R$. Thus, for any cases the ratio $\mathcal{K}_s(d)/\mathcal{K}_s(d_1)$ is bounded by a universal constant depending on $R$ which is independent of $s$. By using \Cref{lem:decreasing}, we arrive at
\begin{equation*}
\frac{(\theta r/2)^2}{\mathcal{I}_0(\theta r/2)} \mathcal{M}^-_{\mathcal{L}_0} w \leq C \frac{\mathcal{I}_0(R)/R^2}{\mathcal{I}_0(\theta r/2)/(\theta r /2)^2} \mathcal{H}(7R) \leq C
\end{equation*}
in $B(x_0, 7(\theta r/14))$.

Let $Q_2 \subset Q_1$ be the dyadic cube of generation $k_{\theta r /14}$ containing the point $x_0$. Then by \Cref{lem:L-eps}, we have
\begin{align*}
| \lbrace u < u(x_0)/2 \rbrace \cap Q_2 |
&= | \lbrace w > \left( \left( 1-\theta/2 \right)^{-s} - 1/2 \right) u(x_0) \rbrace \cap Q_2 | \\
&\leq \frac{C | Q_2 |}{\left( \left( 1-\theta/2 \right)^{-s} - 1/2 \right)^\varepsilon u(x_0)^\varepsilon} \left( \inf_{B(x_0, \delta_1 \theta r/14)} w + C \right)^\varepsilon.
\end{align*}
We can make the quantity $(1-\theta/2)^{-s} - 1/2$ bounded away from 0 by taking $\theta > 0$ sufficiently small. Recalling that $w(x_0) = ((1-\theta/2)^{-s}-1) u(x_0)$, we obtain
\begin{equation*}
| \lbrace u < u(x_0)/2 \rbrace \cap Q_2 | \leq C | Q_2 | \left( ((1-\theta/2)^{-s}-1)^\varepsilon + \left( \frac{C}{u(x_0)} \right)^\varepsilon \right).
\end{equation*}
We choose a constant $\theta > 0$ sufficiently small so that
\begin{equation*}
C \left((1 - \theta/2)^{-s} - 1 \right)^\varepsilon \leq \frac{1}{4}.
\end{equation*}
If $t > 0$ is sufficiently large so that $C (C/u(x_0))^\varepsilon < 1/4$, then we arrive at \eqref{eq:ball_theta}. Therefore, $t$ is uniformly bounded and the desired result follows.
\end{proof}

\section{H\"older estimates} \label{sec:Holder}

In this section, the following H\"older regularity result is proved. \Cref{thm:Holder} follows from simple scaling and covering arguments.

\begin{lemma} \label{lem:Holder_0}
Assume $0 < s_0 \leq s < 1$. There is a universal constant $\varepsilon_0$ such that if $u \in C^2(B_{7R})$ is a function such that $|u| \leq \frac{1}{2}$ in $B_{7R}$ and
\begin{equation*}
\frac{(7R)^2}{\mathcal{I}_0(7R)} \mathcal{M}^+ u \geq - \varepsilon_0 \quad \text{and} \quad \frac{(7R)^2}{\mathcal{I}_0(7R)} \mathcal{M}^- u \leq \varepsilon_0 \quad \text{in} ~ B_{7R},
\end{equation*}
then $u \in C^\alpha$ at $0 \in \Hn$ with an estimate
\begin{equation*}
|u(x) - u(0)| \leq CR^{-\alpha} d_\Hn(x, 0)^\alpha,
\end{equation*}
where $\alpha \in (0, 1)$ and $C > 0$ are universal constants depending only on $n$, $\lambda$, $\Lambda$, $R$, and $s_0$.
\end{lemma}

\begin{proof}
Let $R_k := 7 \cdot 4^{-k} R$ and $B_k := B_{R_k}$. It suffices to construct an increasing sequence $\lbrace m_k \rbrace_{k \geq 0}$ and a decreasing sequence $\lbrace M_k \rbrace_{k \geq 0}$ such that $m_k \leq u \leq M_k$ in $B_k$ and $M_k - m_k = 4^{-\alpha k}$. We initially choose $m_0 = - 1/2$ and $M_0 = 1/2$ for the case $k = 0$. Let us assume that we have sequences up to $m_k$ and $M_k$ and find $m_{k+1}$ and $M_{k+1}$.

For $x \in B_{2R_{k+1}}$, let $Q_1$ be a dyadic cube of generation $k_{R_{k+1}/7}$. In $Q_1$, either $u > (M_k + m_k)/2$ or $u \leq (M_k + m_k)/2$ in at least half of the points in measure. We assume
\begin{equation} \label{eq:Q_1_half}
|\lbrace u > (M_k + m_k)/2 \rbrace \cap Q_1| \geq \frac{1}{2} |Q_1|.
\end{equation}
A function defined by
\begin{equation*}
v(x) := \frac{u(x) - m_k}{(M_k - m_k)/2}
\end{equation*}
satisfies $v \geq 0$ in $B_k$ by the induction hypothesis. To apply \Cref{lem:L-eps}, let us consider a function $w := v^+$, which satisfies
\begin{equation} \label{eq:Q_1_half_w}
|\lbrace w > 1 \rbrace \cap Q_1| \geq \frac{1}{2}|Q_1|
\end{equation}
by \eqref{eq:Q_1_half}. Since $\frac{(7R)^2}{\mathcal{I}_0(7R)} \mathcal{M}^- v \leq 2\varepsilon_0 / (M_k-m_k)$ in $B_{7R}$, we have
\begin{align*}
\frac{R_{k+1}^2}{\mathcal{I}_0(R_{k+1})} \mathcal{M}^- w
&\leq \frac{R_{k+1}^2}{\mathcal{I}_0(R_{k+1})} (\mathcal{M}^- v + \mathcal{M}^+ v^-) \\
&\leq \frac{2\varepsilon_0}{M_k - m_k} \frac{R_{k+1}^2}{\mathcal{I}_0(R_{k+1})} \frac{\mathcal{I}_0(7R)}{(7R)^2} + \frac{R_{k+1}^2}{\mathcal{I}_0(R_{k+1})} \mathcal{M}^+ v^-
\end{align*}
in $B_{3R_{k+1}}$. By \Cref{lem:decreasing}, we have
\begin{equation*}
\frac{R_{k+1}^2}{\mathcal{I}_0(R_{k+1})} \frac{\mathcal{I}_0(7R)}{(7R)^2} \leq \left( \frac{R_{k+1}}{7R} \right)^s = 4^{-(k+1)s} < 4^{-ks_0}.
\end{equation*}
Thus, we obtain
\begin{equation*}
\frac{R_{k+1}^2}{\mathcal{I}_0(R_{k+1})} \mathcal{M}^- w \leq 2\varepsilon_0 + \frac{R_{k+1}^2}{\mathcal{I}_0(R_{k+1})} \mathcal{M}^+ v^-,
\end{equation*}
by assuming $\alpha < s_0$.

For $\mathcal{M}^+v^-$, we use an inequality $v(z) \geq -2((d_\Hn(z,0)/R_k)^\alpha-1)$, $z \in \Hn \setminus B_k$, which follows from the definition of $v$ and the properties of sequences $M_k$ and $m_k$. Then, for any $x_0 \in B_{3R_{k+1}}$, we have
\begin{align*}
\mathcal{M}^+ v^- (x_0)
&\leq \Lambda \int_{\Hn \setminus B_k} v^-(z) \mathcal{K}_s(d_\Hn(x_0, z)) \, \d\mu_{\Hn}(z) \\
&\leq 2\Lambda \int_{\Hn \setminus B_k} \left( \left( \frac{d_\Hn(z,0)}{R_k} \right)^\alpha -1 \right) \mathcal{K}_s(d_\Hn(x_0, z)) \, \d\mu_{\Hn}(z).
\end{align*}
Since $d_\Hn(z,0) \leq 4d_\Hn(z, x_0)$, we obtain
\begin{equation} \label{eq:DCT}
\frac{R_{k+1}^2}{\mathcal{I}_0(R_{k+1})} \mathcal{M}^+ v^- \leq \frac{2\Lambda R_{k+1}^2}{\mathcal{I}_0(R_{k+1})} \int_{\Hn \setminus B(x_0, R_{k+1})} \left( \left( \frac{d_\Hn(z, x_0)}{R_{k+1}} \right)^\alpha -1 \right) \mathcal{K}_s(d_\Hn(x_0, z)) \, \d\mu_{\Hn}(z).
\end{equation}
Let $I$ be the right-hand side of \eqref{eq:DCT}. By the dominated convergence theorem, we know that $I$ converges to 0 as $\alpha \to 0$ for each $s$. Let $\alpha_s > 0$ be the constant such that $I \leq \varepsilon_0$ whenever $\alpha \leq \alpha_s$. Since $I$ is continuous with respect to $\alpha$ and $s$, $\alpha_s$ is chosen continuously. Thus, the quantity $\min_{s \in [s_0,1]} \alpha_s$ is positive and depends on $s_0$ (not on $s$). By choosing $\alpha=\min_{s \in [s_0,1]}\alpha_s$, we obtain
\begin{equation*}
\frac{R_{k+1}^2}{\mathcal{I}_0(R_{k+1})} \mathcal{M}^-w \leq 3\varepsilon_0
\end{equation*}
in $B(x, 7(R_{k+1}/7))$ for $x \in B_{2R_{k+1}}$. Therefore, by \Cref{lem:L-eps} and \eqref{eq:Q_1_half_w}, we have
\begin{align*}
\frac{1}{2} | Q_1 | \leq | \lbrace w > 1 \rbrace \cap Q_1 | \leq C | Q_1 | \left( w(x) + 3\varepsilon_0 \right)^\varepsilon,
\end{align*}
or equivalently, $\theta \leq w(x) + 3\varepsilon_0$ for some universal constant $\theta > 0$. By taking $\varepsilon_0 < \theta/6$, we arrive at $w \geq \theta /2$ in $B_{2R_{k+1}}$. Thus, if we set $M_{k+1} = M_k$ and $m_{k+1} = M_k - 4^{-\alpha(k+1)}$, then 
\begin{equation*}
M_{k+1} \geq u \geq m_k + \frac{M_k - m_k}{4} \theta = M_k - \left(1 - \frac{\theta}{4} \right) 4^{-\alpha k} \geq m_{k+1}
\end{equation*}
in $B_{k+1}$.

When \eqref{eq:Q_1_half} does not hold, a similar proof can be made by using $\frac{(7R)^2}{\mathcal{I}_0(7R)} \mathcal{M}^+ u \geq - \varepsilon_0$ instead of $\frac{(7R)^2}{\mathcal{I}_0(7R)} \mathcal{M}^- u \leq \varepsilon_0$.
\end{proof}

\begin{appendix}

\section{Special functions} \label{sec:special_functions}

The equation
\begin{equation*}
\rho^2 \frac{\d^2 y}{\d \rho^2} + \rho \frac{\d y}{\d \rho} - (\rho^2 + \nu^2)y =0
\end{equation*}
is called the modified Bessel's equation, and its solutions are given by
\begin{equation*}
I_\nu(\rho) = \sum_{j=0}^\infty \frac{1}{j! \Gamma(\nu+j+1)}\left( \frac{\rho}{2} \right)^{2j+\nu} \quad \text{and}\quad K_\nu(\rho) = \frac{\pi}{2} \frac{I_{-\nu}(\rho)-I_\nu(\rho)}{\sin \nu \pi}.
\end{equation*}
They are called modified Bessel functions of the first and second kind, respectively. They satisfy the recurrence relations
\begin{equation*}
K_{\nu+1}-K_{\nu-1}=\frac{2\nu}{R}K_\nu, \quad I_{\nu-1}-I_{\nu+1}=\frac{2\nu}{R} I_\nu,
\end{equation*}
and the following system of first-order differential equations:
\begin{equation} \label{eq:derivative}
\begin{cases}
I_\nu' &=I_{\nu-1}-\frac{\nu}{R}I_\nu, \\
I_\nu' &=I_{\nu+1}+\frac{\nu}{R}I_\nu,
\end{cases}
\quad\text{and}\quad
\begin{cases}
K_\nu' &=-K_{\nu-1}-\frac{\nu}{R}K_\nu, \\
K_\nu' &=-K_{\nu+1}+\frac{\nu}{R} K_\nu.
\end{cases}
\end{equation}
Moreover, the following asymptotic behavior is well known. For further properties of special functions, the reader may consult the book \cite{OLBC10}.

\begin{lemma} \label{lem:asymptotic}
The asymptotic behavior of the modified Bessel functions are given by
\begin{align*}
I_\nu(\rho) &\sim \frac{1}{\Gamma(\nu+1)}\left( \frac{\rho}{2} \right)^\nu, \quad \nu \neq -1, -2, \cdots, \\
K_\nu(\rho) &\sim \frac{1}{2} \Gamma(\nu) \left( \frac{\rho}{2} \right)^{-\nu}, \quad \mathrm{Re}\, \nu > 0,
\end{align*}
as $\rho \to 0$, and
\begin{align*}
I_\nu(\rho) &\sim \frac{e^\rho}{\sqrt{2\pi \rho}}, \\
K_\nu(\rho) &\sim \sqrt{\frac{\pi}{2\rho}} e^{-\rho},
\end{align*}
as $\rho \to \infty$.
\end{lemma}

In this paper, some special functions involving the modified Bessel functions appear. Let us first study the kernel of the fractional Laplacian on hyperbolic spaces.

\begin{lemma} \label{lem:kernel-asymp}
There exist constants $C_{1}, C_{2} > 0$, depending only on $n$, such that
\begin{equation*}
C_{1} \leq \frac{\sinh^{n-1}(R) \mathcal{K}_{n, s, 1}(R)}{s(1-s) R^{-s} I_{n/2-1}(\frac{n-1}{2} R) K_{n/2+s}(\frac{n-1}{2} R)} \leq C_{2}.
\end{equation*}
\end{lemma}

The proof of \Cref{lem:kernel-asymp} is divided into two parts: the odd and even dimensional cases. For the even dimensional case, we need the following lemma from \cite[Lemma~3.5]{KKL22b}.

\begin{lemma}\cite[Lemma~3.5]{KKL22b} \label{lem:even_dim}
  Let $a>0$ and $\nu > -\frac{n-1}{2}$. Then
  \begin{align*}
    \int^{\infty}_{R} \frac{\sinh^{-n/2+1} r}{\sqrt{\cosh r-\cosh R}} r^{-\nu}K_{n/2 + \nu} (ar) \,\d r \sim \sqrt{\frac{\pi}{2}} \frac{\Gamma(\nu+\frac{n-1}{2})}{\Gamma(\nu+\frac{n}{2})} R^{-\nu}\sinh^{-n/2 +1} (R)K_{n/2+\nu} (aR) 
  \end{align*}
  as $R \to 0^{+}$ up to dimensional constants.
\end{lemma}

\begin{proof} [Proof of \Cref{lem:kernel-asymp}]
Observe that the function $\sqrt{R} I_{n/2-1}(\frac{n-1}{2}R)$ is comparable to the function $\sinh^{\frac{n-1}{2}}(R)$ up to dimensional constants by \Cref{lem:asymptotic}. Thus, it suffices to prove that $\sinh^{\frac{n-1}{2}}(R) \mathcal{K}_{n, s, 1}(R)$ is comparable to $s(1-s) R^{-1/2-s} K_{n/2+s}(\frac{n-1}{2}R)$ up to dimensional constants.

Let us first consider the odd dimensional case $n=2m+1$. It is sufficient to prove
\begin{equation} \label{eq:claim-G}
C_{3} \leq G(R, s) := \frac{R^{1/2+s} \sinh^{m}R}{K_{m+1/2+s}(mR)} \left( \frac{-\partial_R}{\sinh R} \right)^{m} \mathscr{K}_{1/2+s, m}(R) \leq C_{4}, \quad R>0, s \in [0,1],
\end{equation}
by recalling \eqref{eq:kernel-odd} and observing $c_{n, s} \leq C(n) s(1-s)$. By \Cref{lem:asymptotic}, the modified Bessel function $K_{\nu}(\rho)$ is asymptotic to $\sqrt{\frac{\pi}{2\rho}} e^{-\rho}$ as $\rho \to \infty$ uniformly with respect to $\nu \in [1/2, n/2+1]$. Moreover, its $i$-th derivative is asymptotic to $\rho^{-1/2}e^{-\rho}$ up to constants depending only on $n$ and $i$ by \eqref{eq:derivative} in the same range of $\nu$. Therefore, $G$ is bounded from above and below near $R = \infty$ by positive constants depending only on $n$.

On the other hand, $G$ is also bounded near $R=0$ by a dimensional constant since $K_{\nu}(\rho)$ is asymptotic to $2^{\nu-1} \Gamma(\nu) \rho^{-\nu}$ as $\rho \to 0$ and $2^{\nu-1} \Gamma(\nu)$ is bounded from above and below by dimensional constants when $\nu \in [1/2, n/2+1]$. Since $G$ is continuous, we conclude \eqref{eq:claim-G}.

Let us next consider the even dimensional case $n=2m$. In this case, we consider
\begin{equation*}
H(R, s) := \int_R^\infty \frac{\sinh r}{\sqrt{\cosh r - \cosh R}} \left(\frac{-\partial_r}{\sinh r} \right)^{m} \mathscr{K}_{\frac{1+2s}{2}, \frac{n-1}{2}}(r) \,\mathrm{d}r
\end{equation*}
We first prove that $R^{1+s} e^{(n-1)R} H(R, s)$ is bounded from above and below near $R=\infty$ by positive constants depending only on $n$. As $R$ is sufficiently close to $\infty$, we have
  \begin{align*}
H(R, s)
    &\le C \int_R^\infty \frac{e^r}{\sqrt{\sinh (\frac{r-R}{2})} \sqrt{\sinh (\frac{r+R}{2})}} r^{-1-s} e^{(-n + 1/2) r} \,\mathrm{d}r  \\                                                                                                                    & \le C \frac{1}{\sqrt{\sinh R}}\int _0^\infty \frac{1}{\sqrt{\sinh \frac{t}{2}}} (t+R)^{-1-s} e^{(-n +3/2)(t+R)} \,\mathrm{d}t \\                                                                                                                   & \le C R^{-1-s} e^{-(n-1)R} \int_0^\infty \frac{1}{\sqrt{\sinh \frac{t}{2}}} dt \\
    &\leq C R^{-1-s} e^{-(n-1)R}
  \end{align*}
and
  \begin{align*}
  H(R, s)
  &= \int_R^\infty 2\sinh r \sqrt{\cosh r - \cosh R} \left(\frac{-\partial_r}{\sinh r} \right)^{\frac{n+2}{2}} \mathscr{K}_{\frac{1+2s}{2}, \frac{n-1}{2}}(r) \,\mathrm{d}r \\
& \ge C \int_R^\infty e^r\sqrt{\sinh \frac{r-R}{2}} \sqrt{\sinh \frac{r+R}{2}} r^{-1-s} e^{-(n+1/2) r} \,\mathrm{d}r \\
& \ge C \sqrt{\sinh R} \int_0^\infty \sqrt{\sinh \frac{t}{2}} (t+R)^{-1-s} e^{-(n-1/2)(t+R)} \,\mathrm{d}t \\
& \ge C R^{-1-s} e^{-(n-1)R} \int _0 ^\infty \sqrt{\sinh \frac{t}{2}} (1+t)^{-2}e^{-(n-1/2)t} \,\mathrm{d}t \\
&\geq C R^{-1-s} e^{-(n-1)R}
  \end{align*}
for some dimensional constants $C>0$.

Finally, we prove that $R^{n+2s} H(R, s)$ is bounded from above and below near $R=0$ by positive dimensional constants. By similar arguments as in the odd dimensional case, the function $H$ is comparable to
\begin{equation*}
\int_R^\infty \frac{\sinh^{-m+1} r}{\sqrt{\cosh r - \cosh R}} r^{-1/2-s}K_{m+1/2+s}\left( \frac{2m-1}{2}r \right) \,\mathrm{d}r,
\end{equation*}
and hence to
\begin{equation*}
R^{-1/2-s} \sinh^{-m+1}(R) K_{m+1/2+s}\left( \frac{2m-1}{2}R \right)
\end{equation*}
by \Cref{lem:even_dim}, up to dimensional constants. The desired result now follows from \Cref{lem:asymptotic}.
\end{proof}      

Another special function involving the modified Bessel functions used in this paper is given as follows: we define the definite integral
\begin{equation} \label{eq:A}
A^{\beta}_{\mu, \nu} = \int \rho^{\beta} I_{\mu} K_{\nu} \,\mathrm{d}\rho.
\end{equation}

\begin{lemma} \label{lem:A}
Let $k \in \mathbb{N}$ and $\beta=\mu-\nu+2k+1 \neq 0, 1, \dots, k$. Then
\begin{equation*}
A^{\beta}_{\mu, \nu} = \sum_{j=0}^{k} \frac{(-1)^{j} k!/(k-j)!}{2(\beta-k)\cdots(\beta-k+j)} \rho^{\beta+1} \left( I_{\mu+j} K_{\nu-j} + I_{\mu+j+1} K_{\nu-j-1} \right).
\end{equation*}
\end{lemma}

\begin{proof}
By using \eqref{eq:derivative} and the integration by parts, we obtain
\begin{align}
&A^{\beta}_{\mu, \nu} = \frac{1}{\beta+\mu-\nu+1} \left( \rho^{\beta+1} I_{\mu} K_{\nu} + A^{\beta+1}_{\mu, \nu-1} - A^{\beta+1}_{\mu+1, \nu} \right) \quad\text{and} \label{eq:A1} \\
&A^{\beta}_{\mu, \nu} = \frac{1}{\beta-\mu+\nu+1} \left( \rho^{\beta+1} I_{\mu} K_{\nu} - A^{\beta+1}_{\mu-1, \nu} + A^{\beta+1}_{\mu, \nu+1} \right). \label{eq:A2}
\end{align}
By plugging \eqref{eq:A2}, with $\mu, \nu$ replaced by $\mu+1, \nu-1$, into \eqref{eq:A1}, we obtain
\begin{equation} \label{eq:A-iteration}
A^{\beta}_{\mu, \nu} = \frac{1}{\beta+\mu-\nu+1} \rho^{\beta+1} (I_{\mu} K_{\nu} + I_{\mu+1} K_{\nu-1}) - \frac{\beta-\mu+\nu-1}{\beta+\mu-\nu+1} A^{\beta}_{\mu+1, \nu-1}.
\end{equation}
The desired result follows by iterating \eqref{eq:A-iteration}.
\end{proof}

We also define the indefinite integral
\begin{equation} \label{eq:A-definite}
A^{\beta}_{\mu, \nu}(R) = \int_{0}^{R} \rho^{\beta} I_{\mu} K_{\nu} \,\mathrm{d}\rho.
\end{equation}
Note that it is well defined by \Cref{lem:asymptotic}, provided that $-\mu \notin \mathbb{N}$, $\mu > 0$, and $\beta+\mu-\nu+1>0$.

\end{appendix}

\section*{Acknowledgement}

The research of Jongmyeong Kim is supported by the National Research Foundation of Korea(NRF) grant funded by the Korea government(MSIT)(2016K2A9A2A13003815). The research of Minhyun Kim is funded by the Deutsche Forschungsgemeinschaft (GRK 2235/2 2021 - 282638148) and the National Research Foundation of Korea (RS-2023-00252297). The research of Ki-Ahm Lee is supported by the Ministry of Education of the Republic of Korea and the National Research Foundation of Korea (RS-2025-00515707).


\begin{thebibliography}{10}

\bibitem{AOCMM21}
D.~Alonso-Or{\'a}n, F.~Chamizo, {\'A}.~D. Mart{\'\i}nez, and A.~Mas.
\newblock Pointwise monotonicity of heat kernels.
\newblock {\em Revista Matem{\'a}tica Complutense}, pages 1--14, 2021.

\bibitem{BGS15}
V.~Banica, M.~d.~M. Gonz\'{a}lez, and M.~S\'{a}ez.
\newblock Some constructions for the fractional {L}aplacian on noncompact
  manifolds.
\newblock {\em Rev. Mat. Iberoam.}, 31(2):681--712, 2015.

\bibitem{Cab97}
X.~Cabr\'{e}.
\newblock Nondivergent elliptic equations on manifolds with nonnegative
  curvature.
\newblock 50(7):623--665.

\bibitem{CS09}
L.~Caffarelli and L.~Silvestre.
\newblock Regularity theory for fully nonlinear integro-differential equations.
\newblock 62(5):597--638.

\bibitem{CS11a}
L.~Caffarelli and L.~Silvestre.
\newblock The {E}vans-{K}rylov theorem for nonlocal fully nonlinear equations.
\newblock {\em Ann. of Math. (2)}, 174(2):1163--1187, 2011.

\bibitem{CS11b}
L.~Caffarelli and L.~Silvestre.
\newblock Regularity results for nonlocal equations by approximation.
\newblock {\em Arch. Ration. Mech. Anal.}, 200(1):59--88, 2011.

\bibitem{CKKW21}
Z.-Q. Chen, P.~Kim, T.~Kumagai, and J.~Wang.
\newblock Heat kernel upper bounds for symmetric {M}arkov semigroups.
\newblock {\em J. Funct. Anal.}, 281(4):Paper No. 109074, 40, 2021.

\bibitem{CKW19}
Z.-Q. Chen, T.~Kumagai, and J.~Wang.
\newblock Elliptic {H}arnack inequalities for symmetric non-local {D}irichlet
  forms.
\newblock 125:1--42.

\bibitem{CY75}
S.~Y. Cheng and S.~T. Yau.
\newblock Differential equations on {R}iemannian manifolds and their geometric
  applications.
\newblock 28(3):333--354.

\bibitem{Chr90}
M.~Christ.
\newblock A {$T(b)$} theorem with remarks on analytic capacity and the {C}auchy
  integral.
\newblock 60/61(2):601--628.

\bibitem{CEMS01}
D.~Cordero-Erausquin, R.~J. McCann, and M.~Schmuckenschl\"{a}ger.
\newblock A {R}iemannian interpolation inequality \`a la {B}orell, {B}rascamp
  and {L}ieb.
\newblock {\em Invent. Math.}, 146(2):219--257, 2001.

\bibitem{Fer15}
M.~Ferreira.
\newblock Harmonic analysis on the {M}\"{o}bius gyrogroup.
\newblock {\em J. Fourier Anal. Appl.}, 21(2):281--317, 2015.

\bibitem{GM96}
W.~Gangbo and R.~J. McCann.
\newblock The geometry of optimal transportation.
\newblock {\em Acta Math.}, 177(2):113--161, 1996.

\bibitem{GGG03}
I.~M. Gelfand, S.~G. Gindikin, and M.~I. Graev.
\newblock {\em Selected topics in integral geometry}, volume 220 of {\em
  Translations of Mathematical Monographs}.
\newblock American Mathematical Society, Providence, RI, 2003.
\newblock Translated from the 2000 Russian original by A. Shtern.

\bibitem{Geo05}
V.~Georgiev.
\newblock {\em Semilinear hyperbolic equations}, volume~7 of {\em MSJ Memoirs}.
\newblock Mathematical Society of Japan, Tokyo, second edition, 2005.
\newblock With a preface by Y. Shibata.

\bibitem{Gig12}
N.~Gigli.
\newblock Second order analysis on {$(\mathscr{P}_2(M),W_2)$}.
\newblock 216(1018):xii+154.

\bibitem{Gri03}
A.~Grigor'yan.
\newblock Heat kernels and function theory on metric measure spaces.
\newblock In {\em Heat kernels and analysis on manifolds, graphs, and metric
  spaces ({P}aris, 2002)}, volume 338 of {\em Contemp. Math.}, pages 143--172.
  Amer. Math. Soc., Providence, RI, 2003.

\bibitem{GHH18}
A.~Grigor'yan, E.~Hu, and J.~Hu.
\newblock Two-sided estimates of heat kernels of jump type {D}irichlet forms.
\newblock {\em Adv. Math.}, 330:433--515, 2018.

\bibitem{GN98}
A.~Grigor'yan and M.~Noguchi.
\newblock The heat kernel on hyperbolic space.
\newblock {\em Bull. London Math. Soc.}, 30(6):643--650, 1998.

\bibitem{Hel08}
S.~Helgason.
\newblock {\em Geometric analysis on symmetric spaces}, volume~39 of {\em
  Mathematical Surveys and Monographs}.
\newblock American Mathematical Society, Providence, RI, second edition, 2008.

\bibitem{JB96}
C.~M. Joshi and S.~K. Bissu.
\newblock Inequalities for some special functions.
\newblock {\em J. Comput. Appl. Math.}, 69(2):251--259, 1996.

\bibitem{Jos17}
J.~Jost.
\newblock {\em Riemannian geometry and geometric analysis}.
\newblock Universitext. Springer, Cham, seventh edition.

\bibitem{KKL22b}
J.~Kim, M.~Kim, and K.-A. Lee.
\newblock The fractional $ p $-{L}aplacian on hyperbolic spaces.
\newblock {\em arXiv preprint arXiv:2210.07029}, 2022.

\bibitem{KKL22a}
J.~Kim, M.~Kim, and K.-A. Lee.
\newblock Harnack inequality for nonlocal operators on manifolds with
  nonnegative curvature.
\newblock {\em Calc. Var. Partial Differential Equations}, 61(1):Paper No. 22,
  29, 2022.

\bibitem{Kim04}
S.~Kim.
\newblock Harnack inequality for nondivergent elliptic operators on
  {R}iemannian manifolds.
\newblock 213(2):281--293.

\bibitem{KKL14}
S.~Kim, S.~Kim, and K.-A. Lee.
\newblock Harnack inequality for nondivergent parabolic operators on
  {R}iemannian manifolds.
\newblock 49(1-2):669--706.

\bibitem{KL14}
S.~Kim and K.-A. Lee.
\newblock Parabolic {H}arnack inequality of viscosity solutions on {R}iemannian
  manifolds.
\newblock {\em J. Funct. Anal.}, 267(7):2152--2198, 2014.

\bibitem{KL13}
Y.-C. Kim and K.-A. Lee.
\newblock Regularity results for fully nonlinear parabolic integro-differential
  operators.
\newblock {\em Math. Ann.}, 357(4):1541--1576, 2013.

\bibitem{KL16}
Y.-C. Kim and K.-A. Lee.
\newblock Cordes-{N}irenberg type estimates for nonlocal parabolic equations.
\newblock {\em Nonlinear Anal.}, 130:330--360, 2016.

\bibitem{KL17}
Y.-C. Kim and K.-A. Lee.
\newblock The {E}vans-{K}rylov theorem for nonlocal parabolic fully nonlinear
  equations.
\newblock {\em Nonlinear Anal.}, 160:79--107, 2017.

\bibitem{Laf91}
A.~Laforgia.
\newblock Bounds for modified {B}essel functions.
\newblock {\em J. Comput. Appl. Math.}, 34(3):263--267, 1991.

\bibitem{MTW05}
X.-N. Ma, N.~S. Trudinger, and X.-J. Wang.
\newblock Regularity of potential functions of the optimal transportation
  problem.
\newblock {\em Arch. Ration. Mech. Anal.}, 177(2):151--183, 2005.

\bibitem{OLBC10}
F.~W.~J. Olver, D.~W. Lozier, R.~F. Boisvert, and C.~W. Clark, editors.
\newblock {\em N{IST} handbook of mathematical functions}.
\newblock U.S. Department of Commerce, National Institute of Standards and
  Technology, Washington, DC; Cambridge University Press, Cambridge, 2010.
\newblock With 1 CD-ROM (Windows, Macintosh and UNIX).

\bibitem{PVM21}
W.~Peng, T.~Varanka, A.~Mostafa, H.~Shi, and G.~Zhao.
\newblock Hyperbolic deep neural networks: A survey, 2021.

\bibitem{Rat06}
J.~G. Ratcliffe.
\newblock {\em Foundations of hyperbolic manifolds}, volume 149 of {\em
  Graduate Texts in Mathematics}.
\newblock Springer, New York, second edition, 2006.

\bibitem{SC92}
L.~Saloff-Coste.
\newblock Uniformly elliptic operators on {R}iemannian manifolds.
\newblock {\em J. Differential Geom.}, 36(2):417--450, 1992.

\bibitem{Seg11}
J.~Segura.
\newblock Bounds for ratios of modified {B}essel functions and associated
  {T}ur\'{a}n-type inequalities.
\newblock {\em J. Math. Anal. Appl.}, 374(2):516--528, 2011.

\bibitem{Son86}
H.~M. Soner.
\newblock Optimal control with state-space constraint. {II}.
\newblock {\em SIAM J. Control Optim.}, 24(6):1110--1122, 1986.

\bibitem{Ter85}
A.~Terras.
\newblock {\em Harmonic analysis on symmetric spaces and applications. {I}}.
\newblock Springer-Verlag, New York, 1985.

\bibitem{Thu97}
W.~P. Thurston.
\newblock {\em Three-dimensional geometry and topology. {V}ol. 1}, volume~35 of
  {\em Princeton Mathematical Series}.
\newblock Princeton University Press, Princeton, NJ, 1997.
\newblock Edited by Silvio Levy.

\bibitem{Vil09}
C.~Villani.
\newblock {\em Optimal transport}, volume 338 of {\em Grundlehren der
  Mathematischen Wissenschaften [Fundamental Principles of Mathematical
  Sciences]}.
\newblock Springer-Verlag, Berlin, 2009.
\newblock Old and new.

\bibitem{WZ13}
Y.~Wang and X.~Zhang.
\newblock An {A}lexandroff-{B}akelman-{P}ucci estimate on {R}iemannian
  manifolds.
\newblock 232:499--512.

\bibitem{Yau75}
S.~T. Yau.
\newblock Harmonic functions on complete {R}iemannian manifolds.
\newblock 28:201--228.

\end{thebibliography}

\end{document}